\documentclass[letterpaper,12pt,reqno,twoside]{amsart}

\usepackage{amssymb}
\usepackage{amsmath}
\usepackage{graphicx}
\usepackage[colorlinks,citecolor=blue,urlcolor=blue]{hyperref}
\usepackage{cite}
\usepackage[hmargin=2cm,vmargin=2.5cm]{geometry}

\newcommand{\be}{\begin{eqnarray}}
\newcommand{\ee}{\end{eqnarray}}

\newcommand{\beq}{\begin{equation}}
\newcommand{\eeq}{\end{equation}}
\newcommand{\beqn}{\begin{equation*}}
\newcommand{\eeqn}{\end{equation*}}

\newcommand{\imag}{\mathrm{i}}

\newcommand{\defas}{\mathrel{\raise.095ex\hbox{$:$}\mkern-4.2mu=}}
\newcommand{\defasr}{\mathrel{=\mkern-4.2mu\raise.095ex\hbox{$:$}}}

\newcommand{\slot}{\,\cdot\,}

\DeclareMathAlphabet{\mathfat}{U}{bbold}{m}{n}          

\DeclareMathOperator{\E}{\mathrm{E}}

\newtheorem{thm}{Theorem}

\newtheorem{cor}[thm]{Corollary}
\newtheorem{lem}[thm]{Lemma}

\newtheorem{remark}[thm]{Remark}

\newcommand\cA{{\mathcal A}}

\newcommand\cC{{\mathcal C}}

\newcommand\cF{{\mathcal F}}

\newcommand\cH{{\mathcal H}}
\newcommand\cI{{\mathcal I}}

\newcommand\cM{{\mathcal M}}
\newcommand\cN{{\mathcal N}}

\newcommand\cP{{\mathcal P}}

\newcommand\cS{{\mathcal S}}

\newcommand\cV{{\mathcal V}}

\newcommand\cZ{{\mathcal Z}}

\newcommand\bC{{\mathbb C}}

\newcommand\bE{{\mathbb E}}

\newcommand\bH{{\mathbb H}}

\newcommand\bN{{\mathbb N}}

\newcommand\bP{{\mathbb P}}

\newcommand\bR{{\mathbb R}}

\newcommand\bZ{{\mathbb Z}}

\newcommand\rE{{\mathrm E}}

\newcommand\rP{{\mathrm P}}

\newcommand\rT{{\mathrm T}}

\newcommand\rd{{\mathrm d}}
\newcommand\rp{{\mathrm p}}

\newcommand\fA{{\mathfrak A}}
\newcommand\fB{{\mathfrak B}}

\newcommand\fF{{\mathfrak F}}

\newcommand{\ve}{\varepsilon}

\def\bfA{\mathbf{A}}
\def\bfB{\mathbf{B}}

\def\bfS{\mathbf{S}}

\def\bfV{\mathbf{V}}

\def\bfc{\mathbf{c}}

\def\bff{\mathbf{f}} 
\def\bfg{\mathbf{g}}

\def\bft{\mathbf{t}}

\def\bfv{\mathbf{v}}

\def\bfSigma{\mathbf{\Sigma}}

\begin{document}

\title[An invariance principle for billiards with random scatterers]{A vector-valued almost sure invariance principle for Sinai billiards with random scatterers}

\author[Mikko Stenlund]{Mikko Stenlund}
\address[Mikko Stenlund]{
Department of Mathematics, University of Rome ``Tor Vergata''\\
Via della Ricerca Scientifica, I-00133 Roma, Italy; Department of Mathematics and Statistics, P.O.\ Box 68, Fin-00014 University of Helsinki, Finland.}
\email{mikko.stenlund@helsinki.fi}
\urladdr{http://www.math.helsinki.fi/mathphys/mikko.html}



\begin{abstract}
Understanding the statistical properties of the aperiodic planar Lorentz gas stands as a grand challenge in the theory of dynamical systems. Here we study a greatly simplified but related model, proposed by Arvind Ayyer and popularized by Joel Lebowitz, in which a scatterer configuration on the torus is randomly updated between collisions. Taking advantage of recent progress in the theory of time-dependent billiards on the one hand and in probability theory on the other, we prove a vector-valued almost sure invariance principle for the model. Notably, the configuration sequence can be weakly dependent and non-stationary. We provide an expression for the covariance matrix, which in the non-stationary case differs from the traditional one. We also obtain a new invariance principle for Sinai billiards (the case of fixed scatterers) with time-dependent observables, and improve the accuracy and generality of existing results.
\end{abstract}

\maketitle


\subsection*{Acknowledgements}
The author wishes to thank Arvind Ayyer and Joel Lebowitz, both for posing the problem several years ago and for many conversations afterwards. The author is grateful to Marco Lenci for informative discussions on the theory of billiards and to Magda Peligrad for correspondence related to probability theory. Also the hospitality of Carlangelo Liverani and Universit\`a di Roma ``Tor Vergata'' during the preparation of this paper are acknowledged. The work was initiated while the author was a Visiting Member at the Courant Institute of Mathematical Sciences, and it has been supported by the Academy of Finland.


\section{Introduction}
Recall that having infinitely many fixed scatterers in a periodic configuration in $\bR^2$ and a particle moving in the exterior of the scatterers, the particle being in free motion up to elastic collisions with the scatterers, is a model often called the planar periodic Lorentz gas. For example, if the configuration is invariant under the translations of $\bZ^2$, the model corresponds to Sinai billiards on the two-dimensional torus. Under certain assumptions, including that the free path of the particle is uniformly bounded from above and below, many statistical limit results are accordingly known to hold true. 
A planar \emph{aperiodic} Lorentz gas is obtained by relaxing the periodicity assumption on the scatterer configuration. For example, one can begin with a periodic configuration, shift each scatterer by a small amount, independently of the others, and fix the resulting configuration for good. Understanding the statistical properties of the particle trajectories in this case is an outstanding problem in modern dynamics. A source of major difficulties is the phenomenon or re-collisions: The billiard particle could hit the same scatterer infinitely many times and, because of that, the scatterer configuration \emph{seen by the particle} at a given time depends in a complicated way on the history of the billiard trajectory. On the other hand, if the configuration has temporal randomness --- say it is refreshed randomly and independently of the past after each collision --- the situation is simpler. Yet, little is presently known even in that setting. The motivation of this paper is to improve the state of the affairs. To that end, we are going to analyze a model of Sinai billiards with random scatterers, in which the configuration is randomly updated between collisions.

Consider, then, the following random billiard table on a two-dimensional torus obtained by identifying the opposite sides of the unit square. First, a disk of radius $\bar R$ is placed on the torus with its center located at the corner of the square, after which it is fixed for good. We will refer to this disk as the ``gray disk'' in the future. Next, another disk of radius $R$, called the ``white disk'', is placed on the torus at a \emph{random} location so that its center is within distance $\ve$ from the center of the square; see Figure~\ref{fig:table}. We will assume throughout that the ``no-overlap'' conditions
\beqn
\max(\bar R,R+\ve) < \tfrac12
\quad
\text{and}
\quad
\bar R+R+\ve < \tfrac{1}{\sqrt 2}
\eeqn
hold. Given such a triple $(\bar R,R,\ve)$, the strict inequalities guarantee a uniform positive lower bound on the distance between the disks, regardless of the centering of the white disk.

The rules of the dynamics on the table are as follows. Consider a particle traveling with unit speed in the complement of the disks on the surface of the torus. If the particle hits the \emph{white} disk, it bounces elastically off the inert boundary of the disk and continues its motion with unit speed, as illustrated in Figure~\ref{fig:table}. If, on the other hand, the particle hits the \emph{gray} disk instead, it behaves in exactly the same way, but the location of the \emph{white} disk is \emph{refreshed randomly}, so that its center remains within distance $\ve$ from the center of the square. Finally, corresponding to the sides of the original unit square, there are four ``transparent walls'' (the opposite pairs of which are identified) terminating on the gray disk. If the particle meets a transparent wall, it simply passes through. However, if it makes a ``clean pass'' in that it avoids hitting the gray walls at the two ends of the transparent wall in question, see Figure~\ref{fig:clean_pass} and its caption, the location of the \emph{white} disk is again \emph{refreshed randomly}. In brief, the dynamics is that of ordinary billiards, except that the location of the white disk is refreshed randomly each time the particle hits the gray disk or makes a clean pass through a transparent wall. The type of randomness we assume will be made precise in Section~\ref{sec:Result}. In particular, the sequence of configurations is allowed to be weakly dependent and non-stationary.

\begin{figure}[!ht]
\begin{center}
\includegraphics[width=0.3\linewidth]{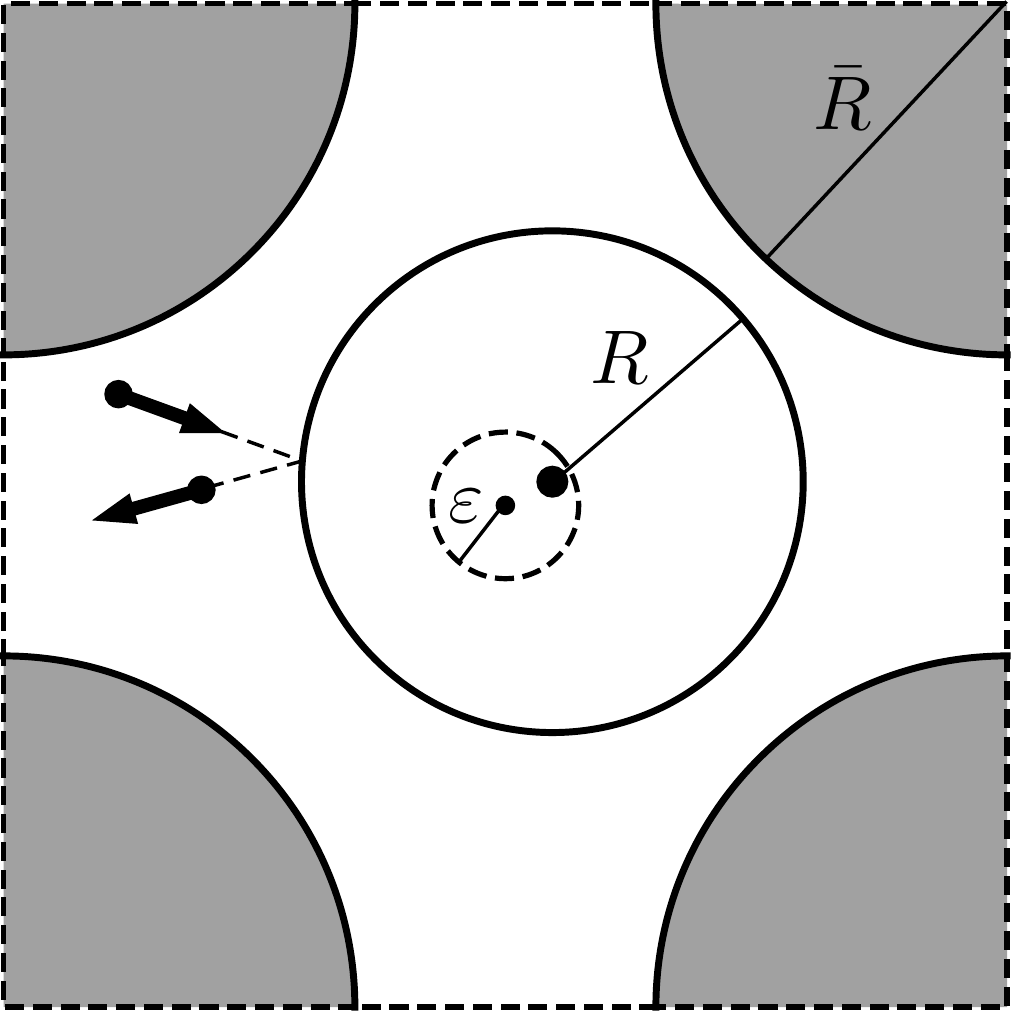}
\caption{The random billiard table. A two-dimensional torus is obtained by identifying the opposite sides of the unit square shown. The gray disk of radius $\bar R$ is fixed for good with its center at the corner. The white disk has a fixed radius $R$ and a random center within distance $\ve$ from the center of the square. The random centering is refreshed after each collision with the gray disk or a clean pass through a transparent wall; see text for details.}
\label{fig:table}
\end{center}
\end{figure}

\begin{figure}[!ht]
\begin{center}
\includegraphics[width=0.3\linewidth]{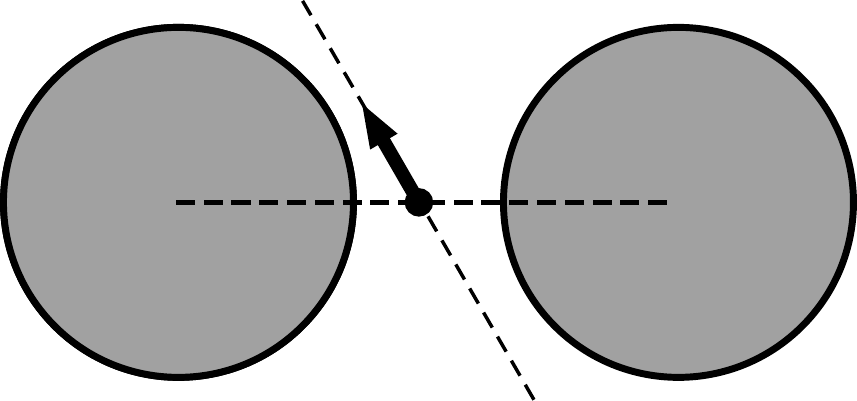}
\caption{A velocity vector corresponding to a ``clean pass'' through a transparent wall. A pass is clean if the segment of unit length parallel to the velocity vector and centered at the point of crossing does not intersect the gray disk.}
\label{fig:clean_pass}
\end{center}
\end{figure}

Notice that the no-overlap conditions with their strict inequalities yield a uniform positive lower bound on length of the free path of the particle. It will be necessary to bound the length of the free path (between successive collisions with a disk) from above. To that end, we introduce the so-called ``finite horizon'' conditions on the geometry of the disks that we will assume to be in effect at all times. With the aid of Figure~\ref{fig:horizon}, it is easy to verify that the condition
\beqn
\bar R \geq \tfrac{1}{2\sqrt 2}
\eeqn
guarantees there is no diagonal passage along which the particle could escape without ever colliding with a disk. On the other hand, the condition
\beqn
\bar R + R - \ve \geq \tfrac12
\eeqn
guarantees there is neither vertical nor horizontal passage; the white disk will then, in any allowed position, contain the smaller centered disk shown in Figure~\ref{fig:horizon}. In fact, under these conditions the free path is uniformly bounded above.
\begin{figure}[!ht]
\begin{center}
\includegraphics[width=0.3\linewidth]{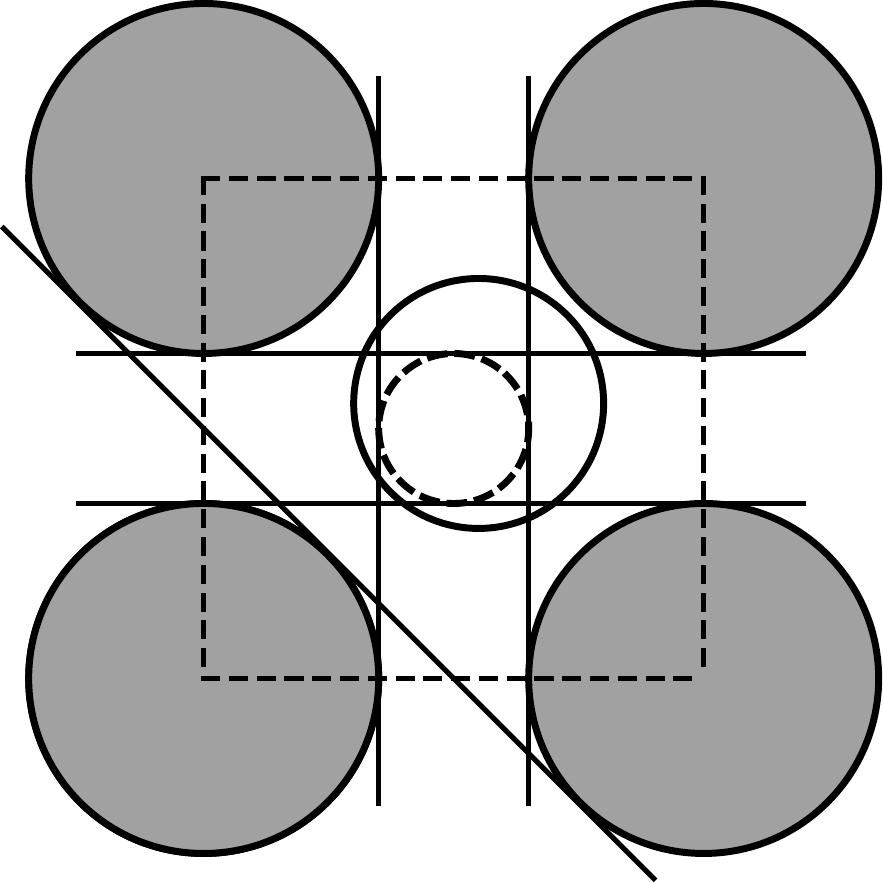}
\caption{The finite horizon assumption. If the gray disk has radius $\bar R\geq \tfrac{1}{2\sqrt 2}$, the diagonal passage is blocked. Given $\bar R$, the white disk blocks the horizontal and vertical passages if it contains the centered dashed disk of radius $\tfrac12-\bar R$.}
\label{fig:horizon}
\end{center}
\end{figure}

Of course, if the above conditions are satisfied for some value of $\ve$, given $\bar R$ and $R$, they remain satisfied for any smaller value of $\ve$ as well. The idea is that $\bar R$ and $R$ are considered fixed once and for all and, for all sufficiently small values of $\ve$, the statistical properties of the motion of the particle can be analyzed.

To the author's knowledge, the problem studied here was initially proposed by Arvind Ayyer and then Joel Lebowitz. A similar model was studied in~\cite{Troubetzkoy_1991}. It was also put forward in \cite{AyyerLiveraniStenlund_2009}, and the recent \cite{Nandori_2012} mentions a reminiscent model due to Lebowitz. For a toy version, some theorems were obtained in~\cite{AyyerLiveraniStenlund_2009} and~\cite{AyyerStenlund}. See also~\cite{SimulaStenlund_2009,LeskelaStenlund_2011} for models with some similarities.  Below we prove an \emph{almost sure invariance principle for vector-valued observables} for the model at hand. 
%
\medskip

As the aim of the paper is to demonstrate a phenomenon, above all, we have not shied away from making convenience assumptions on the geometry at the expense of generality. See, however, Remark~\ref{rem:generalizations} for some immediate generalizations of no extra cost.

The analysis below relies on several standard constructions in the theory of dispersing billiards. These include homogeneous local stable and unstable manifolds, together with a good understanding of their properties, among others. In particular, the random-scatterer model we consider here differs from the classical Sinai billiards in which the configuration of scatterers remains forever fixed. The excellent book \cite{ChernovMarkarian_2006} contains a detailed account of the classical theory and the paper \cite{StenlundYoungZhang_2012} the necessary \emph{uniform} time-dependent generalizations. (Also \cite{OttStenlundYoung,Stenlund_2011} might interest the reader, although the setup there is simpler than in \cite{StenlundYoungZhang_2012}.) It is also necessary to make the observation that, although a billiard map is traditionally defined as the first-return map of the billiard flow to the solid boundary of the domain in question, the theory extends without difficulty to cover (flat or curved) transparent walls as cross sections \cite{Lenci_2003,Lenci_2006}. 
For the sake of clarity, we will not dwell on these quite involved but well-understood constructions here, instead referring the interested reader to the references cited. The mandatory deviation from this occurs when we introduce the billiard maps: They need to be defined with due care, for the presence of flat (here transparent) walls does have the tendency to render the billiard maps non-uniformly hyperbolic. In our case the system \emph{remains uniformly hyperbolic} by a careful choice of the cross section. Other than that, the focus in the present paper is solely on establishing the estimates required to prove the desired invariance principle.

An almost sure invariance principle for Sinai billiards (fixed scatterers) was first obtained for scalar-valued observables in \cite{Chernov_2006,ChernovMarkarian_2006,MelbourneNicol_2005} and for vector-valued ones in \cite{MelbourneNicol_2009}. The approach in the present paper --- which in particular applies to the setting of fixed configurations of those works --- is based on verifying the conditions of \cite{Gouezel_2010} on the characteristic function of the vector-valued process. As the title suggests, the latter paper has been tailored for dynamical systems specifically with the transfer operator formulation in mind. For billiard models the spectral method~\cite{DemersZhang_2011} is quite technical and abstract due to the well-known issue of singularities, on top of the presence of a contracting direction. Here we take an alternative route, providing sufficient information on the characteristic function via strong bounds on correlation functions. In this sense, our work is related to~\cite{Stenlund_2010} and~\cite{Chernov_2006}, and also underlines the possibility of implementing \cite{Gouezel_2010} in situations where the operator setup may be either impractical or out of reach. It is conceivable that the problem could (with an unclear amount of work) also be approached by constructing a random Young tower.

Let us point out that an almost sure invariance principle for vector-valued observables is a nontrivial improvement on the scalar case. Although it is a standard to deduce a multi-dimensional \emph{central} limit theorem, where convergence takes place in distribution, from the corresponding one-dimensional result using the Cram\'er--Wold theorem, the same cannot be said about \emph{almost sure} invariance principles.

Note that there is no physical continuous-time flow associated to the model, because of the way the configuration is randomly updated after each return: A set of (non-interacting) test particles will always ``see'' the same sequence of configurations, which in continuous time would mean that they always return to the section simultaneously. (In particular, the roof function of the corresponding suspension flow is constant.) More realistic flows in the setting of moving scatterers will be future work, extending the results of~\cite{StenlundYoungZhang_2012}.


\bigskip
\noindent{\bf How the paper is organized.}
In Section~\ref{sec:Model} we give a mathematical description of the model. Once the necessary concepts have been introduced, we present the main results of the paper in Section~\ref{sec:Result}. There we also discuss the approach used in the proofs. Section~\ref{sec:preliminaries} is devoted to the technical preliminaries that are required for the analysis of the problem. In Section~\ref{sec:Proof} we prove the main results by establishing sufficient bounds on certain correlation functions and by analyzing the structure of the covariance matrix.



\section{Mathematical description of the model}\label{sec:Model}
In this section we describe the model in full detail so that the main theorem can be formulated in the next section.


\subsection{The free zone condition}

In addition to the no-overlap and finite horizon conditions, we assume the ``free zone'' condition
\beqn
R+\ve < \frac{(1-2\bar R)^{\frac12} - \bar R(1-2\bar R)}{2(1-\bar R)} \equiv L
\eeqn
on $(\bar R,R,\ve)$. By elementary geometrical computations, the free zone condition can equivalently be stated as follows:
\begin{quote}
Given the triple $(\bar R,R,\ve)$, the white disk does not intersect any chord originating from a transparent wall and terminating on a piece of the gray disk at the end of the same transparent wall, in any possible location of the white disk.
\end{quote}
This is clarified in Figure~\ref{fig:free_zone}. The role of the free zone condition is one of convenience, and it could be done away with. As our objective is not to be as general as possible (see Introduction), we choose to keep it. Under the free zone condition, the billiard trajectory will never hit the white disk more than once between successive returns to the cross section $\cM$ to be defined below.

\begin{figure}[!ht]
\begin{center}
\includegraphics[width=0.3\linewidth]{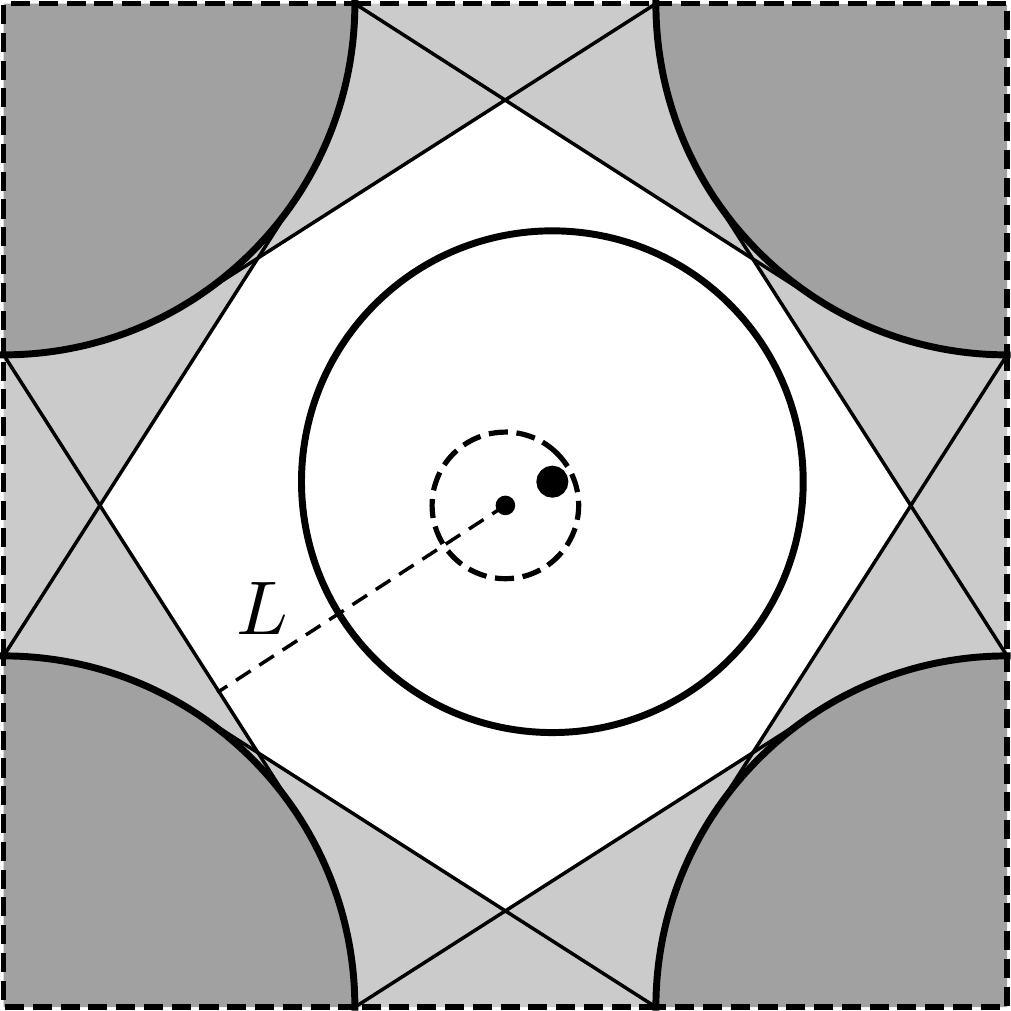}
\caption{The free zone condition. The straight lines are tangent to the gray disk. The white disk is not allowed to enter the region with a light gray shading. The distance from the center of the square to the latter region is~$L$.}
\label{fig:free_zone}
\end{center}
\end{figure}

In what follows, we will assume that the triple $(\bar R,R,\ve)$ has been fixed so that all the conditions introduced so far hold.


\subsection{The cross section}
Here we define the cross section, $\cM$, of the ``full'' billiard flow phase space with respect to which we are eventually going to define the billiard maps. The cross section itself will be independent of the location of the white disk. For that reason the white disk will not enter the discussion at this stage.

We begin by labeling the non-random walls $\Gamma_1,\dots,\Gamma_8$ of the domain as shown in Figure~\ref{fig:section}. The boundary of the gray disk forms the four solid walls $\Gamma_i$ with an odd index $i$, and the other four with an even index are the transparent walls. The walls are ``closed'' in that they contain their endpoints. As customary, let the position~$r$ on a wall be parametrized by arclength and denote by~$\varphi$ the clockwise angle relative to the normal vector of the wall pointing into the domain. Facing in the direction of the normal vector, the value of $r$ increases from left to right. Then any unit vector based on any of the walls and pointing into the domain corresponds, in a one-to-one manner, to an element $x=(r,\varphi)$ of the disjoint union
\beqn
\widehat{\cM} = \coprod_{i=1}^8 \Gamma_i\times [-\tfrac\pi2,\tfrac\pi2].
\eeqn

The mentioned cross section,~$\cM$, of the billiard flow phase space will be obtained by deleting a part of~$\widehat\cM$ as follows. Pick an arbitrary transparent wall $\Gamma_i$, $i$ even, and consider an arbitrary point $x = (r,\varphi)\in\Gamma_i\times[-\tfrac\pi2,\tfrac\pi2]$ representing a velocity vector on $\Gamma_i$. If $x$ corresponds to a crossing which does not yield a clean pass (see Figure~\ref{fig:clean_pass}) as described in the Introduction, we delete $x$ from the coordinate rectangle $\Gamma_i\times[-\tfrac\pi2,\tfrac\pi2]$. Such values of $x$ are illustrated in Figure~\ref{fig:section}. The remaining parts of the coordinate rectangle form a region reminiscent of a rhombus, whose closure we denote by~$\cM_i$. We have thus defined
\beqn
\cM_i\subsetneq \Gamma_i\times[-\tfrac\pi2,\tfrac\pi2], \quad \text{$i$ even.}
\eeqn
On solid walls we include all velocity vectors, thereby setting
\beqn
\cM_i = \Gamma_i\times[-\tfrac\pi2,\tfrac\pi2], \quad \text{$i$ odd.}
\eeqn
Finally, the cross section itself is defined as
\beqn
\cM = \coprod_{i=1}^8 \cM_i \subset \widehat{\cM}.
\eeqn

In plain words, each collision with the gray disk marks a return to $\cM$, in that the post-collision velocity vector is an element of $\cM$. On the other hand, a crossing of a transparent wall is a return to $\cM$ if and only if it yields a clean pass (see Figure~\ref{fig:clean_pass}). Thus, the three crossings of $\Gamma_6$ shown in Figure~\ref{fig:section} are \emph{not} returns to $\cM$.

It is clear (see Figure~\ref{fig:section}) that there exists a constant $d>0$ such that 
\beq\label{eq:transparent_cosine}
\text{$\cos\varphi\geq d$\qquad if \quad $x=(r,\varphi)\in\cM$ is on a transparent wall.}
\eeq
Indeed, the interpretation of the above discussion is that the particle is stopped ($x\in\cM$) at a crossing of a transparent wall only when the crossing angle is sufficiently far from being parallel (depending on position) to the transparent wall. In the opposite case the particle is allowed to pass through without registering a return.

\begin{figure}[!ht]
\begin{center}
\includegraphics[height=0.35\linewidth]{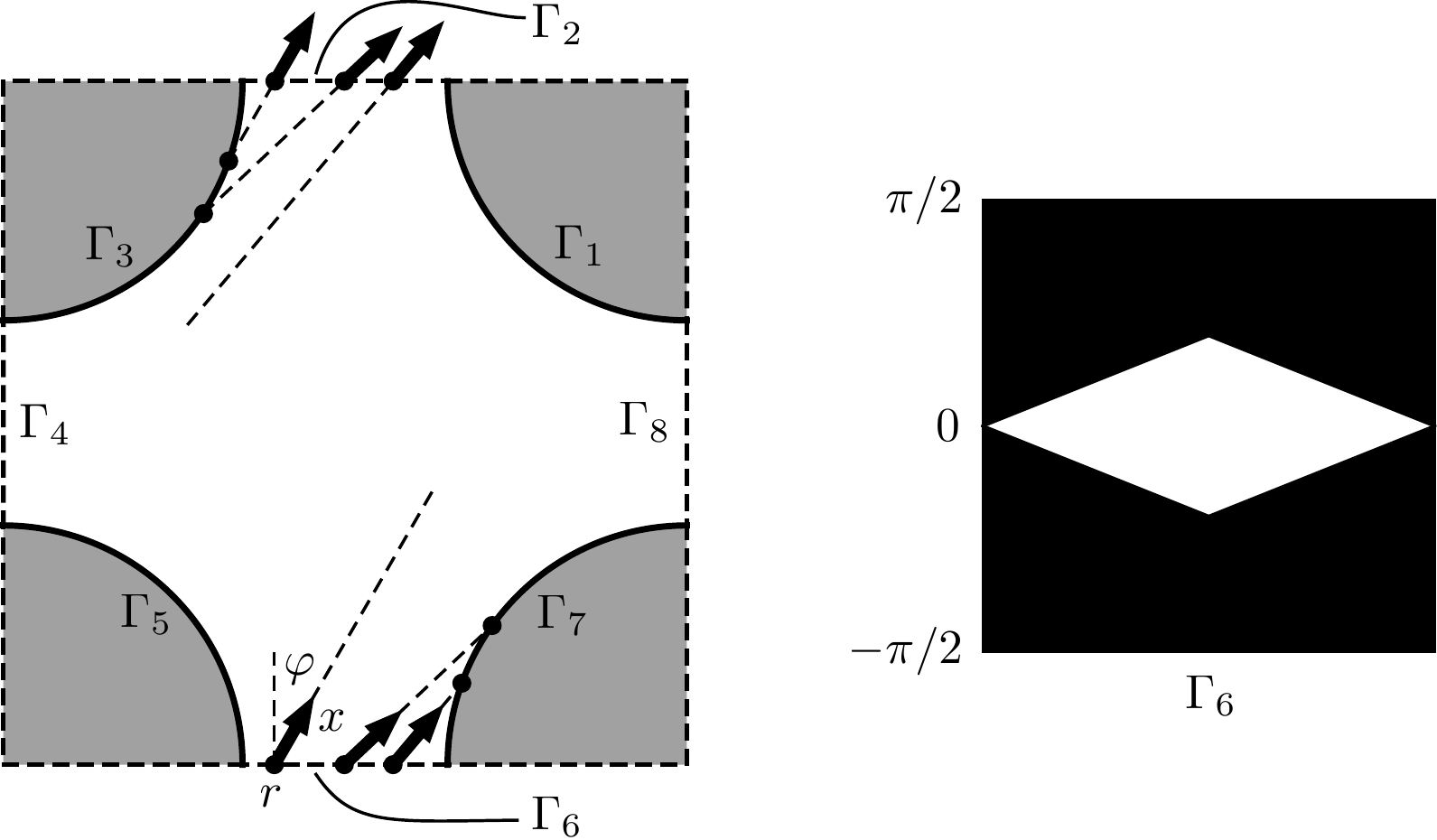}
\caption{The cross section~$\cM$ of the billiard flow phase space does \emph{not} contain the shown vectors on $\Gamma_6$, because they correspond to ``immediate'' collisions (dots on the solid walls) with the gray disk. Deleting all such points from the coordinate rectangle leaves us with the white rhombus-like region $\cM_6\subset\cM$ (sketched above), corresponding to all clean passes through the transparent wall (Figure~\ref{fig:clean_pass}).}
\label{fig:section}
\end{center}
\end{figure}


\subsection{The billiard maps}
Let $\bfc$ denote the position vector of the center of the white disk relative to the center of the square. Given any $\bfc$ and $\delta>0$, we will throughout the text denote by $B_\delta(\mathbf{c})$ the $\delta$-neighborhood of $\bfc$. We say that $\bfc$ is admissible, if $\bfc\in B_\ve(\mathbf{0})$. Having defined the cross section~$\cM\subset\widehat\cM$, the billiard map $F_\bfc$ is defined simply as the first return map from the cross section~$\cM$ to itself of the billiard flow $\Phi_\bfc$ corresponding to the white disk centered at~$\bfc$. If $T_\bfc:\cM\to\bR$ stands for the return time function, then
\beqn
F_\bfc = \Phi_\bfc^{T_\bfc} : \cM\to\cM.
\eeqn
The map $F_\bfc$ is called admissible if $\bfc$ is admissible.
Notice that, due to the geometry, $T_\bfc(x)$ is defined and uniformly bounded above for all $x\in\cM$. Also $F_\bfc(x)$ is well defined, save for relatively few values of $x$ for which the billiard trajectory starting from $x$ meets a corner point, resulting in conflicting candidates for the value of~$F_\bfc(x)$, one on each of the three walls meeting at the corner. The conflict is resolved by identifying the three values. Such exceptional values of $x$ are examples of singularities, and we will come back to their role shortly. For now, it is harmless to ignore them.

It will be technically beneficial to view the map $F_\bfc$ in a different way. Namely, notice that the billiard trajectory between any $x\in\cM$ and its image $F_\bfc(x)\in\cM$ either misses the white disk completely or hits it precisely once. Now, let $F_\bfc^*$ be the billiard map which, in addition to $\cM$, counts collisions with the white disk as returns. More precisely, let $\cM^*$ be the enlarged cross section $\cM^* = \cM\,\amalg\, {\Gamma^*\times[-\pi/2,\pi/2]}$, where $\Gamma^*$ is a parametrization of the boundary of the \emph{white} disk. The parametrization can be fixed in such a way that $\cM^*$ is independent of $\bfc$, and this is what we do. Then $F_\bfc^*:\cM^*\to\cM^*$ is the first return map to~$\cM^*$ of the billiard flow, defined in a fashion similar to $F_\bfc$. Next, let $n_\bfc:\cM^*\to\{1,2\}$ be the smallest number of returns to $\cM^*$ which yields a return to $\cM$, i.e.,
\beqn
n_\bfc(x) = \min\bigl\{n\geq 1\,:\,(F_\bfc^*)^n(x)\in\cM\bigr\}, \quad x\in\cM^*.
\eeqn
It is clear that $n_\bfc(x)=2$ if $x\in\cM$ and the billiard trajectory connecting $x$ to $F_\bfc(x)$ experiences a collision with the white disk. Otherwise $n_\bfc(x) = 1$. With these definitions,
\beqn
F_\bfc = (F_\bfc^*)^{n_\bfc} |_\cM.
\eeqn
Figure~\ref{fig:rules} illustrates the situation.

\begin{figure}[!ht]
\begin{center}
\includegraphics[width=0.35\linewidth]{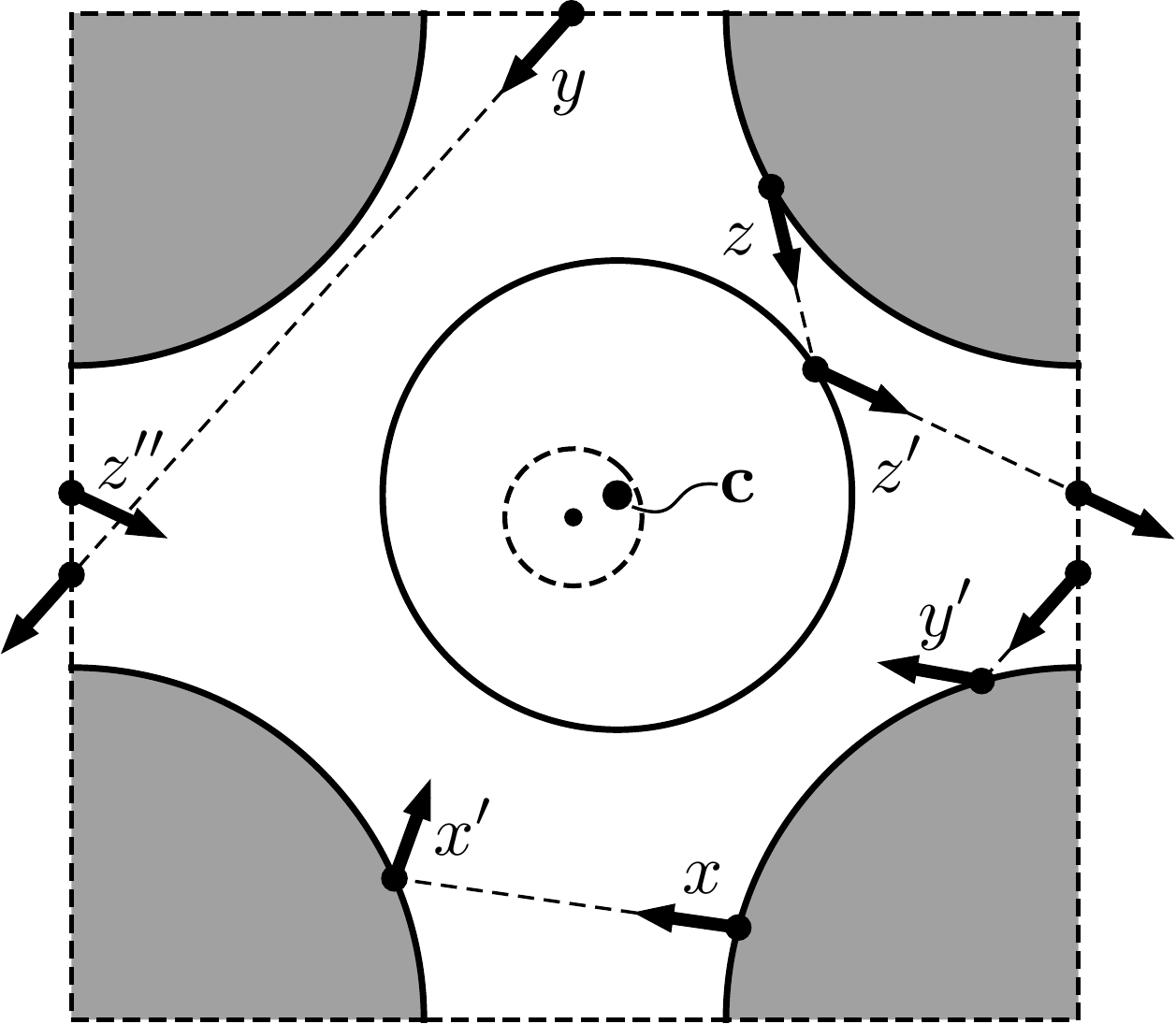}
\caption{The action of the billiard map $F_\bfc$ corresponding to the white disk being centered at $\bfc$. We have $x' = F_\bfc(x)$ and $y' = F_\bfc(y)$. On the other hand, $z' = F_\bfc^*(z)$ and $z'' = F_\bfc^*(z') = F_\bfc(z)$.}
\label{fig:rules}
\end{center}
\end{figure}

We will use the shorthand notation $\{n_\bfc = n\} \equiv \{x\in\cM^*\,:\,n_\bfc(x) = n\}$ with $n=1,2$. Notice that
\beqn
\cM^* = \{n_\bfc = 1\} \cup \{n_\bfc = 2\} 
\eeqn
and
\beq\label{eq:n_c=2}
\{n_\bfc = 2\} = (F_\bfc^*)^{-1}(\cM^*\setminus\cM) \subset \cM
\eeq
hold true by construction, and that
\beqn
F_\bfc^*(\{n_\bfc = 2\}) \subset \{n_\bfc = 1\}.
\eeqn

\subsection{Invariant measure}
We may interpret any $x\in\cM$ as the velocity vector of a billiard particle. Reversing the direction of the velocity to its opposite amounts to mapping $x$ to another point $\cI(x)\in\cM$. It is easy to check that the map $\cI$ is an involution on $\cM$, i.e., it is bijective (taking the earlier identifications at the corners into account) and $\cI^2$ is the identity map of $\cM$. 

Picking an arbitrary element $y\in\cM$, it is possible to trace the billiard trajectory leading to $y$ \emph{backwards}, by following the trajectory starting from $\cI(y)$ \emph{forwards}. By construction, the latter will eventually ``make a return'' to $\cM$, precisely at $F_\bfc(\cI(y))$. Reversing the direction of the velocity once more, we obtain the element $x = \cI(F_\bfc(\cI(y)))\in\cM$. It is clear that $F_\bfc(x) = y$ and that $x$ is the unique preimage of $y$. This is to say that the billiard map $F_\bfc:\cM\to\cM$ is bijective, the inverse billiard map being given by 
\beq\label{eq:inverse}
F_\bfc^{-1} = \cI\circ F_\bfc\circ \cI.
\eeq
It is also a standard fact that~$F_\bfc = \Phi_\bfc^{T_\bfc}$ preserves the probability measure
\beqn
\rd\mu(r,\varphi) = M^{-1} \cos\varphi\,\rd r\,\rd\varphi
\eeqn
on~$\cM$, where $M = \int_\cM \cos\varphi\,\rd r\,\rd\varphi$ is the normalizing factor, for all admissible $\bfc$. (This is a consequence of the general fact that Hamiltonian flows preserve the Liouville measure.) We point out that both of these properties fail generically if the cross section is constructed on the boundaries of \emph{moving} scatterers; see~\cite{StenlundYoungZhang_2012} for examples. 
This is the first of the two fundamental reasons we defined the cross section~$\cM$ in the specific way that we did, and resorted to using non-moving transparent walls. The second one is related to guaranteeing uniform hyperbolicity of the maps~$F_\bfc$, as explained in Section~\ref{sec:hyperbolicity}; see Remark~\ref{rem:non-uniform} in particular.


\section{Main results}\label{sec:Result}
Given a sequence $(\bfc_n)_{n\geq 0}$ of admissible centerings,
we write 
\beq\label{eq:composition}
\cF_n(x) = \cF_n((\bfc_n)_{n\geq 0},x) = F_{\bfc_{n-1}}\circ\dots\circ F_{\bfc_0}(x), \qquad x\in\cM,
\eeq
for short, with the convention that $\cF_0$ is the identity map $\mathrm{id}_\cM$ of $\cM$. The $\bR^d$-valued almost sure invariance principle has to do with approximation of the  sums 
\beqn
\sum_{i=0}^{n-1} \bfA_i\bigl((\bfc_j)_{j\geq 0},x\bigr) = \sum_{i=0}^{n-1}\bff\bigl(\bfc_i,\cF_i\bigl((\bfc_j)_{j\geq 0},x\bigr)\bigr), \qquad n\geq 0,
\eeqn
suitably scaled, by Brownian paths in $\bR^d$, given an $\bR^d$-valued observable $\bff$ together with a random sequence of admissible centerings $\bfc_n$ of the white disk. The strength of the statement is in that the approximation converges \emph{almost surely}. Our result will concern, possibly non-stationary, random sequences $(\bfc_n)_{n\geq 0}$, where the random variables~$\bfc_n$ can be weakly dependent. Before stating the theorem, we introduce the assumptions on the sequences~$(\bfc_n)_{n\geq 0}$.

First, we recall a notion of weak dependence from probability theory.  Suppose a probability space and two sub-sigma-algebras $\fA$ and $\fB$ are given. The so-called maximal correlation coefficient of $\fA$ and $\fB$ is given by 
\beqn
\rho(\fA,\fB) = \sup\bigl\{|\operatorname{Corr}{(f,g)}|\,:\,f\in L^2(\fA;\bR),\,g\in L^2(\fB;\bR)\bigr\}.
\eeqn
It can be shown \cite{Bradley_2005,Withers_1981} that
\beq\label{eq:rho}
\rho(\fA,\fB) = \sup \left\{\frac{|\E(fg)-\E(f)\E(g)|}{\|f\|_2\|g\|_2}\,:\,f\in L^2(\fA;\bC),\,g\in L^2(\fB;\bC)\right\}.
\eeq
(Here $L^2(\fA;\bR)$ is the space of real-valued, $\fA$-measurable, square-integrable functions, etc.)
Note that the supremum in~\eqref{eq:rho} is taken over \emph{complex}-valued functions and that $0\leq \rho(\fA,\fB)\leq 1$.
Consider next a random sequence $(X_n)_{n\geq 0}$ and the probability space of its trajectories. For $n\geq m\geq 0$, let $\fF_m^n$ be the sub-sigma-algebra generated by the variables $X_m,\dots,X_n$. The so-called rho-mixing coefficients are
\beqn
\rho(k) = \sup_{n\geq 0}\rho(\fF_0^n,\fF_{n+k}^\infty), \quad k\geq 0.
\eeqn
The sequence $(X_n)_{n\geq 0}$, or its probability distribution, is called \emph{rho-mixing}, if $\lim_{k\to\infty}\rho(k) = 0$. Obviously, a sequence of independent random variables is rho-mixing with $\rho(k)=0$ for $k\geq 1$. More generally, the same is true of $M$-dependent sequences, for any $M\in\bN$. A stationary Markov chain is rho-mixing if it has the $L^2$-spectral-gap property, and in this case~$\rho(k)$ tends to zero exponentially~\cite{Rosenblatt_1971}. Moreover, if a Markov chain, stationary or not, satisfies $\rho(k_0)<1$ for some $k_0\geq 1$, then $\rho(k)$ tends to zero exponentially~\cite{Rosenblatt_1971,Bradley_2005}. 

\medskip

Given $\ve>0$ and a probability distribution $\bP$ of the random sequence~$(\bfc_n)_{n\geq 0}$ on~$(B_\varepsilon(\mathbf{0}))^\bN$ endowed with the product Borel sigma-algebra $\fF$, our first assumption is the following.

\medskip
\noindent{\bf(A1)}~{\it $\bP$ is exponentially rho-mixing. That is, there exist $c>0$ and $B>0$ such that 
\beqn
\rho(k)\leq Be^{-ck},\qquad k\geq 0.
\eeqn}

Assumption {\bf(A1)} is sufficient in the stationary case. In order to deal with non-stationary sequences, we impose a further condition which holds automatically in the stationary context. It is related to the concept of \emph{asymptotic \underline{mean} stationarity} (see, e.g., \cite{Gray_2009} and the proof of Lemma~\ref{lem:coboundary} below), which means that the measures $\bP_k \equiv \tfrac{1}{k}\sum_{j=0}^{k-1} (\sigma^j)_*\,\bP$, $k\geq 1$, obtained by \emph{averaging} over shifted sequences, tend to a measure $\bar\bP$ which is (necessarily) stationary. Here and below, $\sigma$ is the left shift on $(B_\varepsilon(\mathbf{0}))^\bN$.

Given a measurable function $g:(B_\varepsilon(\mathbf{0}))^\bN\to\bR$, let us introduce the shorthand notation
\beq\label{eq:mean}
\langle g \rangle_k \equiv \frac{1}{k}\sum_{j=0}^{k-1}\bE (g\circ\sigma^j), \qquad k\geq 1,
\eeq
where $\bE$ is the expectation relative to $\bP$. Our second assumption is the following.

\medskip
\noindent{\bf(A2)}~{\it For any bounded measurable function~$g$, there exists a number $\langle g \rangle_\infty\in\bR$ such that $\lim_{k\to\infty} \langle g \rangle_k = \langle g\rangle_\infty$. There exist sequences $(C_m)_{m\geq 0}$ and $(r_k)_{k\geq 0}$ of positive numbers with $r_k\to 0$ such that, for any $m\geq 0$ and any bounded~$\fF_0^m$-measurable function~$g_m$, 
\beq\label{eq:AMS_rate}
| \langle g_m\circ \sigma^\ell \rangle_k - \langle g_m \rangle_\infty | \leq C_m\,r_k\,\|g_m\|_\infty, \qquad k\geq 1,\,\ell \geq 0.
\eeq
Finally, for any $\beta>0$, there exists a constant $D_\beta>0$ such that
\beq\label{eq:AMS_fast}
\sum_{m=0}^{\beta^{-1}\log n} C_m\, r_{n-m} \, e^{-\beta m} \leq \frac{D_\beta \log n}{n}
\eeq
holds whenever $n>\beta^{-1}\log n$ (so that $r_{n-\beta^{-1}\log n}$ on the left side makes sense).}
\medskip

We make the remark about assumption {\bf(A2)} that if $\langle g_m\circ\sigma^\ell \rangle_k$ and $\langle g_m \rangle_k$ converge as $k\to\infty$, they obviously converge to the same limit (Lemma~\ref{lem:AMS}) for all $\ell\geq 0$, but~\eqref{eq:AMS_rate} is needed for a uniform rate of convergence and~\eqref{eq:AMS_fast} for the rate to be sufficiently fast.
\medskip

We are in position to state our main result. As usual, also our billiard maps have singularities consisting of a finitely many smooth curves in $\cM$. For a (non-random) $\bfc$, denote by $\cS_\bfc$ the singularity set of $F_\bfc$; see Section~\ref{sec:sing}. For later convenience, we switch to writing $\omega = (\omega_n)_{n\geq 0}$ for a realization of the random sequence~$(\bfc_n)_{n\geq 0}$. That is, $\omega_n$ denotes the realized position of the center of the white disk at time $n$.

\begin{thm}\label{thm:ASIP}
There exists a number $\ve>0$ such that the following assumptions have the following consequences.

\medskip
\noindent -- \textit{\textbf{Assumptions:}} Let  $\bP$ be a probability distribution on~$(B_\varepsilon(\mathbf{0}))^\bN$ satisfying {\bf(A1)} and {\bf(A2)}. Write $\rP=\bP\otimes\mu$ and $\rE(\slot) = \int(\slot)\,\rd\rP$. Let $\Omega_0\subset B_\varepsilon(\mathbf{0})$ be a measurable set with $\bP((\Omega_0)^\bN)=1$.
Let $\bff:\Omega_0\times\cM\to\bR^d$ (any $d\geq 1$) be a bounded measurable function such that, for all $\bfc\in\Omega_0$, $\bff(\bfc,x)$ is defined for $\mu$-almost-every $x\in\cM$ and $\int \bff(\bfc,x)\,\rd\mu(x)=\mathbf{0}$. Let the maps $\bff(\bfc,\slot)$, $\bfc\in\Omega_0$, be uniformly piecewise H\"older continuous in the sense that there exist $\gamma>0$ and $C_\bff>0$ for which
\beqn
|\bff(\bfc,x)-\bff(\bfc,y)| \leq C_\bff\,d(x,y)^\gamma
\eeqn
for all $x,y$ in the same component of $\cM\setminus\cS_{\bfc}$, for all $\bfc\in\Omega_0$. 
Denote  $\bfA_k\bigl(\omega,x\bigr) = \bff\bigl(\omega_k,\cF_k\bigl(\omega,x\bigr)\bigr)$.

\pagebreak[2]

\medskip
\noindent -- \textit{\textbf{Consequences:}}
\begin{enumerate}
\item 
The formula 
\beq\label{eq:covariance_formula}
\begin{split}
\bfSigma^2 & \equiv
\lim_{k\to\infty}\frac{1}{k} \sum_{\ell=0}^{k-1} \int \bigl(\bfA_0\otimes\bfA_0\bigr)(\sigma^\ell\omega,x) \,\rd\rP(\omega,x)
\\
&\qquad + \sum_{m=1}^{\infty} \, \lim_{k\to\infty}\frac{1}{k} \sum_{\ell=0}^{k-1} \int \bigl(\bfA_0\otimes\bfA_{m} + \bfA_{m}\otimes\bfA_0\bigr)(\sigma^\ell\omega,x)\,\rd\rP(\omega,x)
\end{split}
\eeq
yields a well-defined, symmetric, semi-positive-definite, $d\times d$ matrix~$\bfSigma^2$.
\medskip
\item The matrix $\bfSigma^2$ is the limit covariance of $\frac1{\sqrt n}\sum_{k=0}^{n-1}\bfA_k$. That is, 
\beqn
\lim_{n\to\infty} \frac1n\,\rE\!\left(\sum_{k=0}^{n-1}\bfA_k \otimes \sum_{k=0}^{n-1}\bfA_k\right)=\bfSigma^2.
\eeqn
\medskip
\item The random variables $\frac1{\sqrt n}\sum_{k=0}^{n-1}\bfA_k$ converge in distribution, as $n\to\infty$, to a centered $\bR^d$-valued normal random variable with covariance $\bfSigma^2$.
\medskip
\item Given any $\lambda>\frac14$, there exists a probability space together with two $\bR^d$-valued processes $(\bfA^*_n)_{n\geq 0}$ and $(\bfB_n)_{n\geq 0}$ on it, for which the following statements are true:
\medskip
\begin{enumerate}
\item $(\bfA_n)_{n\geq 0}$ and $(\bfA^*_n)_{n\geq 0}$ have the same distribution.
\medskip
\item The random variables $\bfB_n$, $n\geq 0$, are independent, centered, and normally distributed with covariance $\bfSigma^2$.
\medskip
\item Almost surely, $|\sum_{k=0}^{n-1}\bfA^*_k - \sum_{k=0}^{n-1}\bfB_k| = o(n^\lambda)$.  \footnote{As is customary, given a sequence $(a_n)_{n\geq 0}$, the notation $a_n = o(n^\lambda)$ means that $\lim_{n\to\infty} n^{-\lambda} a_n = 0$.}
\end{enumerate}
\end{enumerate}
\end{thm}

Item (3) of the theorem is called the \emph{averaged (or annealed) central limit theorem} and
item (4) the \emph{vector-valued almost sure invariance principle with covariance~$\bfSigma^2$ and error exponent~$\lambda$}. The ``almost surely'' in item (c) refers to the probability space on which the processes~$(\bfA^*_n)_{n\geq 0}$ and~$(\bfB_n)_{n\geq 0}$ are defined. Since $\sum_{k=0}^{n-1}\bfB_k$ has the interpretation of the location of an $\bR^d$-valued Brownian motion at time $n$, the result is indeed about Brownian approximations of the sequences $(\sum_{k=0}^{n-1}\bfA_k)_{n\geq 0}$. Accordingly, \cite{Strassen_1964,Billingsley_convergence,PhilippStout_1975} discuss several implications of the almost sure invariance principle, and \cite{LaceyPhilipp_1990} shows that the almost sure central limit theorem is among them.

Note that the centerings $\bfc_n$ in the theorem could for example be drawn independently, from the same distribution; the distribution could be uniform in the disk of radius $\ve$, which in a sense corresponds to maximal randomness, or it could be a delta-distribution corresponding to a completely fixed white disk as in Sinai billiards. 

Regarding the regularity condition imposed on the observable, the inter-return flight time and vector displacement of the particle lifted to the plane both satisfy it \cite{ChernovMarkarian_2006}.

The expression for the covariance in~\eqref{eq:covariance_formula} is interesting in that it applies to the \emph{non-stationary} case. In addition to the expectations with respect to the distribution $\rP = \bP\otimes\mu$, the terms in the infinite series entail averaging over the trajectories $(\sigma^\ell \omega)_{\ell\geq 0}$ of the sequences~$\omega$ under the left shift~$\sigma$. In the stationary case the averaging is redundant and the expression of the covariance simplifies to the well-known series $\rE(\bfA_0\otimes\bfA_0)+ \sum_{m=1}^{\infty} \rE (\bfA_0\otimes\bfA_{m} + \bfA_{m}\otimes\bfA_0)$. The matrix~$\bfSigma^2$ can be degenerate for a given observable~$\bff$. To characterize such a situation, we prove the next lemma in Section~\ref{sec:proof_coboundary}, following the classical works of Robinson~\cite{Robinson_1960} and Leonov~\cite{Leonov_1961}. 
\begin{lem}\label{lem:coboundary}
Let $\bff$ and $\bfSigma^2$ be as in Theorem~\ref{thm:ASIP}. Denote by $\Phi$ the skew-product map defined as $\Phi(\omega,x) = (\sigma\omega,F_{\omega_0}(x))$. The matrix $\bfSigma^2$ is degenerate if and only if there exists a (constant) vector $\bfv\in\bR^d$ and a measurable function $g:(\Omega_0)^\bN\times\cM \to \bR$ such that, for all $j\geq 0$, the map $(\omega,x)\mapsto g(\sigma^j\omega,x)$ belongs to $L^2(\rP)$,
$
\lim_{k\to\infty}\frac1k \sum_{j=0}^{k-1} \int | g(\sigma^j\omega,x) |^2\,\rd \rP(\omega,x) < \infty,
$
and the identity
$
\bfv^\rT \bff = g-g\circ \Phi
$
holds except on a measurable set $E\subset (\Omega_0)^\bN\times\cM$ which is asymptotically negligible in the sense that $\lim_{k\to\infty}\frac1k \sum_{j=0}^{k-1}\rP(\sigma^{-j}E) = 0$.
\end{lem}

Despite the if-and-only-if statement, Lemma~\ref{lem:coboundary} is not very practical, because $g$ is only known to belong to the large class $L^2(\rm P)$. It would be interesting to obtain a Livschitz (Liv\v{s}ic) type rigidity result guaranteeing some degree of regularity of $g$; see, e.g.,~\cite{HollandMelbourne_2007,MelbourneNicol_2009,NicolPersson_2012}.

\begin{remark}\label{rem:generalizations}
Theorem~\ref{thm:ASIP} does not depend on the assumption that the scatterers are disk-shaped. The same proof continues to work for more general, strictly convex, shapes of the fixed ``gray'' scatterer and the moving ``white'' scatterer having at least $C^3$-smooth boundaries, because \cite{StenlundYoungZhang_2012} applies to such geometric objects just the same. Of course, it is still necessary that the free path length between the solid scatterers is uniformly bounded from above and below by positive numbers. (This was guaranteed by the no-overlap and finite horizon conditions above.) In order not to have to upgrade the proof, the billiard trajectory should also not hit the white scatterer more than once (or, a bit more generally, a uniformly bounded number of times) between successive returns to the cross section~$\cM$. (This was guaranteed by the free zone condition.) The last thing one has to make sure of is the preservation of uniform hyperbolicity despite the transparent walls by a suitable choice of~$\cM$; between successive \emph{returns} involving one to a transparent wall, the flight time and the angle of the velocity at the transparent wall should be uniformly bounded. (This followed from the clean pass condition; see~\eqref{eq:transparent_cosine}.) Finally, besides just translations, rotations as well as changes to the shape of the white scatterer can be allowed, provided they are uniformly sufficiently small, because \cite{StenlundYoungZhang_2012} applies equally to that scenario. We leave the formulation to the reader.
\end{remark}


The next result concerning \emph{fixed} scatterer configurations (Sinai billiards) and \emph{random observables} can be deduced from the very same proof as Theorem~\ref{thm:ASIP}, simply switching to the familiar cross section $\cN$ corresponding to returns of the billiard flow to the solid boundaries of the scatterers. Since the cross section to be used is clear from the beginning, we formulate the theorem in a rather general geometry (not assuming circular boundaries). We also assume a very modest amount of regularity of the observables, namely dynamical H\"older continuity in the stable and unstable directions separately. (In particular, ``ordinary'' H\"older continuity suffices. See Section~\ref{sec:proof_defs} and Remark~\ref{rem:fixed_dyn_Holder} there for a discussion of dynamical H\"older continuity.)


\begin{thm}\label{thm:fixed_ASIP}
Consider a Sinai billiard on the two-dimensional torus with strictly convex scatterers having $C^3$-smooth boundaries. Assume the free path length is bounded from above and below by positive numbers.
Let $F:\cN\to\cN$ denote the usual billiard map; see the paragraph above. Let~$\Omega_0$ be an index set, $d>0$ a fixed integer, and $\bff:\Omega_0\times\cN\to\bR^d$ a bounded measurable map with the property that (each vector component of) the map $\bff(\omega_0,\slot):\cN\to\bR^d$ is dynamically H\"older continuous on homogeneous local stable and unstable manifolds, with uniform parameters for all $\omega_0\in\Omega_0$. Let  $\bP$ be a probability distribution on the space $(\Omega_0)^\bN$ of sequences~$\omega$ satisfying~{\bf(A1)} and~{\bf(A2)}.  Denote $\bfA_k(\omega,x) = \bff(\omega_k,F^k(x))$, $k\geq 0$. Also write $\rP=\bP\otimes\mu$ and~$\rE(\slot)=\int(\slot)\,\rd\rP$. Then the consequences listed in Theorem~\ref{thm:ASIP} and Lemma~\ref{lem:coboundary} hold.
\end{thm}

The author is unaware of other limit theorems for Sinai billiards with time-dependent (here random) observables. Turning to the special case of a fixed observable (that is, $\Omega_0$ contains a single element), the above theorem enlarges the class of observables for which a weak or almost sure invariance principle is known~\cite{Chernov_2006,MelbourneNicol_2005,MelbourneNicol_2009}; the central limit theorem for the same class was proved in~\cite{Stenlund_2010}. A virtue of the result is that the class of observables in question is convenient to work with when it comes to the study of statistical properties; see discussion below. In addition, Theorem~\ref{thm:fixed_ASIP} improves the known error exponent for scalar-valued observables (from $\tfrac38+\epsilon$ to $\tfrac 14+\epsilon$) as well as for vector-valued observables in dimensions $d\geq 2$ (from $\tfrac {2d+3}{4d+7}+\epsilon$ again to $\tfrac 14+\epsilon$)~\cite{MelbourneNicol_2009}. 

Giving a separate proof of Theorem~\ref{thm:fixed_ASIP} would be pointless, because that would amount to repeating the exact proof of Theorem~\ref{thm:ASIP} with the simplification that the map and the function spaces would no longer depend on time. Accordingly, it is also possible to formulate a more elaborate (albeit general) version of Theorem~\ref{thm:ASIP}, in which the observables belong to time-dependent spaces of dynamically H\"older continuous functions (see Section~\ref{sec:proof_defs}); after all, the proof itself is written precisely for such spaces. But for the sake of a more intelligible statement, and to stress what is really relevant and new, we offer Theorem~\ref{thm:ASIP} in its current form.

The proof of Theorem~\ref{thm:ASIP} benefits from a condition formulated recently by Gou\"ezel under which the vector-valued almost sure invariance principle holds for a sequence of random variables. We will recall the condition in Section~\ref{sec:ASIP_proof}. Gou\"ezel's proof relies on a sophisticated variation of a classical blocking argument, in which the marginals of the random process are grouped into blocks separated by a time gap of length $k$. His condition quantifies a sufficient rate for two blocks to become independent, as their separation $k$ increases, for the vector-valued almost sure invariance principle to hold. 

The strategy of verifying Gou\"ezel's condition in the proof of Theorem~\ref{thm:ASIP} involves a mixture of ideas related to \cite{Stenlund_2010,StenlundYoungZhang_2012}. Viewing the model under study in a way amenable to analysis, including how we defined the cross section~$\cM$ and the billiard maps $F_\bfc$ on it, is also very much part of the method. (See Section~\ref{sec:preliminaries} for more on the importance of setting up the problem carefully.) In \cite{Stenlund_2010}, the function spaces considered in \cite{ChernovMarkarian_2006} were enlarged by relaxing the regularity assumptions. To include \emph{less regular observables} in the analysis of the statistical properties of billiards turned out to be fruitful for proving limit theorems. It was shown that, for observables in these enlarged spaces, a pair correlation bound alone implies the central limit theorem. In the present paper we take advantage of similar spaces, but this time in a 
\emph{time-dependent setup}, and prove bounds on correlation functions which imply Gou\"ezel's condition for the invariance principle. The theory of time-dependent billiards has been developed in \cite{StenlundYoungZhang_2012} to the stage that the existence and uniformity of hyperbolic structures, for example, can be taken for granted, \emph{given the deliberate way in which we defined the cross section~$\cM$ and the billiard maps~$F_\bfc$}. Whereas the general paper \cite{StenlundYoungZhang_2012} did not focus on \emph{random} compositions of billiard maps let alone on their limit theorems, the latter form exactly the goal we are now shooting for.


\section{Preliminaries}\label{sec:preliminaries}
In this section we introduce some concepts and facts that enter the proof of Theorem~\ref{thm:ASIP}. We are not going to present proofs or new results, save for Lemma~\ref{lem:hyperbolic}. Due to the slightly unusual choice of~$\cM$, we need to verify the uniform hyperbolicity of the maps $F_\bfc$ --- in the form of Lemma~\ref{lem:hyperbolic} --- by hand. After this has been achieved, we are immediately in position to apply the machinery of time-dependent billiards developed in~\cite{StenlundYoungZhang_2012}.

\subsection{Singularities}\label{sec:sing}
Like in the case of classical billiards, the maps~$F_\bfc$ have singularities. The qualitative behavior of a billiard map changes in arbitrarily small neighborhoods of its singularities. Given a centering $\bfc$, we call
\beqn
\cS_\bfc = F_\bfc^{-1}(\partial\cM) \cup (F_\bfc^*)^{-1}(\partial\cM^*\setminus\partial\cM)
\eeqn
the singularity set of $F_\bfc$. The first member of the union corresponds to those values of $x\in\cM$ for which the billiard trajectory from $x$ to $F_\bfc(x)$ either hits an endpoint of a wall (corner); meets the gray disk or a transparent wall tangentially; or crosses a transparent wall in such a way that it is a return to~$\cM$, but there are arbitrarily small perturbations of the trajectory which are not returns to~$\cM$. The second member of the union corresponds to those $x\in\cM$ for which the trajectory meets the white disk tangentially. We point out that~\eqref{eq:n_c=2} yields $(F_\bfc^*)^{-1}(\partial\cM^*\setminus\partial\cM)\subset \cM$, which is to say that $\cS_\bfc$ as defined above is a subset of $\cM$ as it should. 
 
The reader will notice that
\beqn
\cS_\bfc = (F_\bfc^*|_{\cM})^{-1}(\partial\cM^*) \cup (F_\bfc^*|_\cM)^{-1}(F_\bfc^*|_{\cM^*\setminus\cM})^{-1}(\partial\cM),
\eeqn
which is a convenient way of viewing the singularity set. The interpretation is that the second member of the union corresponds to ``secondary'' singularities --- singularities as first seen from the white disk and then pulled back to $\cM$.

The singularity set $\cS_\bfc\subset\cM$ consists of piecewise smooth curves with uniformly bounded negative slopes. (In fact, they are stable curves, i.e., curves~$S$ whose tangent spaces $T_x S$, $x\in S$, are contained in the respective stable cones $\cC^s_x$ as defined below.) The smooth curves, or branches, either terminate on other branches in the interior of $\cM$, or they extend all the way to the boundary of $\cM$, and every branch is a part of a path of branches reaching from one part of the boundary to another, as is well known. When $\bfc$ varies, the singularity set $\cS_\bfc$ varies in a continuous manner: Its topology can change in that new branches can appear and existing ones can disappear, and the branches generally undergo deformations, but $\cS_\bfc$ is contained in an arbitrarily small tubular neighborhood of $\cS_{\tilde\bfc}$ provided $\bfc$ is in a sufficiently small neighborhood of~$\tilde\bfc$. 
Moreover, the number of branches is uniformly bounded from above. See~\cite{StenlundYoungZhang_2012} for more on time-dependent singularity sets. In particular, it is possible to prove~\cite{StenlundYoungZhang_2012} that
\begin{lem}\label{lem:map_cont}
Given an admissible $\tilde\bfc$ and a compact subset $E\subset \cM\setminus \cS_{\tilde\bfc}$, there exists $\delta>0$ such that the maps $(\bfc,x)\mapsto F_\bfc(x)$ and $(\bfc,x)\mapsto D_xF_\bfc$ are uniformly continuous on $E \times  B_\delta(\tilde\bfc)$. 
\end{lem}
The map $(\bfc,x)\mapsto F_\bfc(x)$ is thus continuous on the set $\{(\bfc,x)\in B_\ve(\mathbf{0})\times\cM\,:\,x\in\cM\setminus\cS_{\bfc}\}$.

\medskip
We can extend the concept of singularities for sequences of billiard maps in a natural way. Suppose, then, that $\bfc_0$ and $\bfc_1$ are given centerings. The set
\beqn
\cS_{\bfc_0,\bfc_1} = \cS_{\bfc_0} \cup F_{\bfc_0}^{-1}(\cS_{\bfc_1})
\eeqn
is the singularity set of the composition $F_{\bfc_1}\circ F_{\bfc_0}$.
More generally, given a (finite or infinite) sequence $(\bfc_{i})_{i=0}^{N-1}$, we can define inductively, for $1\leq n\leq N$, the singularity set of the composition $F_{\bfc_{n-1}}\circ\dots\circ F_{\bfc_0}$ as
\beqn
\cS_{\bfc_0,\dots,\bfc_{n-1}} = \cS_{\bfc_0} \cup F_{\bfc_0}^{-1}(\cS_{\bfc_1,\dots,\bfc_{n-1}}).
\eeqn
The importance of these sets lies in the fact that the actual billiard trajectory from $x$ to $F_{\bfc_{n-1}}\circ\dots \circ F_{\bfc_0}(x)$ meets a singularity at some point if and only if $x\in \cS_{\bfc_0,\dots,\bfc_{n-1}}$.

Going backward in time, similar singularity sets can be defined for the compositions $(F_{\bfc_{n-1}}\circ\dots\circ F_{\bfc_0})^{-1}$. We denote them $\cS_{\bfc_0,\dots,\bfc_{n-1}}^-$.


\subsection{Uniform hyperbolicity}\label{sec:hyperbolicity}

Let $x = (r,\varphi)\in \cM$ be arbitrary. We denote by $\kappa_\text{gray}$ and $\kappa_\text{white}$ the curvature of the boundary of the gray and the white disk, respectively.  If $x$ is on a solid wall (of the gray disk), define
\beq\label{eq:solid_bounds}
a(x) = \kappa_\text{gray} \quad\text{and}\quad b(x) = \kappa_\text{gray} + \frac{\cos\varphi}{\tau_{\min}}.
\eeq
On the other hand, if $x$ is on a transparent wall, define
\beq\label{eq:cone_slopes}
a(x) = \frac{d}{\tau_{\max} + \frac{1}{\kappa_{\min}}}  \quad\text{and}\quad b(x) =  \frac{\cos\varphi}{\tau_{\min}},
\eeq
where $\kappa_{\min} = \min(\kappa_\text{gray},\kappa_\text{white})$. Note that $b(x)$ is strictly positive by~\eqref{eq:transparent_cosine}. In fact, there exists constants $a_{\min}$ and $b_{\max}$ such that
\beqn
0<a_{\min}\leq a(x) < b(x) \leq b_{\max} <\infty
\eeqn
for all $x\in\cM$.

\begin{lem}\label{lem:hyperbolic}
For any admissible $\bfc$, the unstable cones
\beqn
\cC^u_x = \{(\rd r,\rd\varphi)\in T_x\cM \,:\, a(x) \leq \rd \varphi/\rd r\leq b(x)\}, \quad x\in\cM,
\eeqn
are invariant for $F_\bfc$ and the stable cones
\beqn
\cC^s_x = \{(\rd r,\rd\varphi)\in T_x\cM \,:\, -a(x) \leq\rd \varphi/\rd r\leq -b(x)\}, \quad x\in\cM,
\eeqn
are invariant for $F_\bfc^{-1}$. There exists constants $C>0$ and $\Lambda>1$ such that, given $n\geq 1$ admissible centerings $\bfc_0,\dots,\bfc_{n-1}$, the bound
\beqn
\|D_x (F_{\bfc_{n-1}}\circ\dots \circ F_{\bfc_0})\, v\| \geq C\Lambda^n \|v\|
\eeqn
holds for all $v\in\cC^u_x$ and for all $x\in\cM\setminus\cS_{\bfc_0,\dots,\bfc_{n-1}}$, and the bound
\beqn
\|D_x (F_{\bfc_{n-1}}\circ\dots \circ F_{\bfc_0})^{-1}\, v\| \geq C\Lambda^n \|v\|
\eeqn
holds for all $v\in\cC^s_x$ and for all $x\in\cM\setminus\cS_{\bfc_0,\dots,\bfc_{n-1}}^{-1}$. 
\end{lem}

The lemma states that the maps are \emph{uniformly} hyperbolic with common families of unstable and stable cones and with a uniform minimal expansion rate in the cones. We are therefore in a setting to which the time-dependent analysis that was carried out in~\cite{StenlundYoungZhang_2012} applies directly.

\begin{proof}[Proof of Lemma~\ref{lem:hyperbolic}]
Consider a point $x_1 = (r_1,\varphi_1) \in\cM^*$ and set $x_0 = (r_0,\varphi_0) = (F_\bfc^*)^{-1}(x_1)\in\cM^*$ which corresponds to a straight flight from $x_0$ to $x_1$ (possibly passing through a transparent wall in a way which does not constitute a return to $\cM$). We have \cite{ChernovMarkarian_2006}
\beq\label{eq:slope_formula}
\cV_1 = \kappa_1 + \frac{\cos\varphi_1}{\tau_0 + \frac{\cos\varphi_0}{\cV_0 + \kappa_0}}.
\eeq
Here $\tau_0$ is the path length from $x_0$ to $x_1$, $\kappa_i = \kappa(x_i)$ is the curvature of the wall at $x_i$, $\cV_0$ is the slope of a tangent vector $(\rd r,\rd\varphi)$ at $x_0$ and $\cV_1$ is the slope of the corresponding tangent vector at $x_1$ obtained via the tangent map. Observe that, by our conventions of defining the cross section, we have the uniform bounds
\beq\label{eq:path_bounds}
0<\tau_{\min}\leq \tau_0\leq \tau_{\max}<\infty. 
\eeq

First, we see that the ``positive'' unstable cones $\{(\rd r,\rd\varphi)\in T_x\cM^* \,:\, 0 \leq \rd \varphi/\rd r\leq \infty\}$, $x\in\cM^*$, are invariant for $F_\bfc^*$. In fact, the image of a ``positive'' tangent vector at $x_0$ satisfies
\beqn
\kappa_1\leq \cV_1 \leq \kappa_1 + \frac{\cos\varphi_1}{\tau_{\min}}.
\eeqn

We wish to improve the lower bound on $\cV_1$ in the case $\kappa_1=0$, i.e., when $x_1$ is on a transparent wall and in $\cM$. In this case $x_0$, defined above, is necessarily on a solid wall. Thus, $\kappa_0\geq \kappa_{\min}>0$ and
\beqn
\frac{\cos\varphi_0}{\cV_0 + \kappa_0} \leq \frac{1}{\kappa_{\min}}.
\eeqn
Recalling~\eqref{eq:transparent_cosine}, we have $\cos\varphi_1\geq d$, so that, by~\eqref{eq:slope_formula} and~\eqref{eq:path_bounds},
\beqn
\cV_1 = \frac{\cos\varphi_1}{\tau_0 + \frac{\cos\varphi_0}{\cV_0 + \kappa_0}} \geq \frac{d}{\tau_{\max} + \frac{1}{\kappa_{\min}}} .
\eeqn
Let us extend the definition of $\cC^u_x$ from $x\in\cM$ to all $x\in\cM^*$ by defining~$a(x)$ and~$b(x)$ as in~\eqref{eq:solid_bounds} for $x\in\cM^*\setminus\cM$. By the above analysis, we conclude that $F_\bfc^*$ preserves the obtained cones $\cC^u_x$, $x\in \cM^*$.

Finally, suppose $x$ is a general element of $\cM$ and set $x_1 = F_\bfc(x)\in\cM$. Then the billiard trajectory from $x$ to $x_1$ can be broken into straight legs, the last one of which is from $x_0 = (F_\bfc^*)^{-1}(x_1)\in\cM^*$ to $x_1$. Since each of the legs preserves the cones~$\cC^u_x$, $x\in\cM^*$, the map~$F_\bfc$ preserves the cones~$\cC^u_x$, $x\in\cM$.

Coming to the stable cones, we resort to a time-reversal argument. Recall the expression~\eqref{eq:inverse} of the inverse billiard map $F_\bfc^{-1}$ in terms of the involution~$\cI$. As can be checked, the derivative~$D_x \cI$ equals either $\bigl(\begin{smallmatrix}1 & 0 \\ 0 & -1 \end{smallmatrix}\bigr)$ or $\bigl(\begin{smallmatrix}-1 & 0 \\ 0 & 1 \end{smallmatrix}\bigr)$ depending on whether $x$ is on a solid or a transparent wall, respectively. Using the obvious identity $D_x\cI(\cC^s_x) = \cC^u_{\cI (x)}$ and the above fact that $F_\bfc$ preserves the unstable cones, one confirms that~$F_\bfc^{-1}$ indeed preserves the stable cones.

Next, we prove the uniform expansion property of the compositions $F_{\bfc_{n-1}}\circ\dots\circ F_{\bfc_0}$ on unstable vectors. The complementary statement concerning stable vectors and compositions of inverse maps is similar. To that end, fix a $y_0\in\cM$ and denote $y_{k} = F_{\bfc_{k-1}}\circ\dots\circ F_{\bfc_0}(y_0)\in\cM$ for $1\leq k\leq n$. We assume $y_0\notin \cS_{\bfc_0,\dots,\bfc_{n-1}}$. Also fix an unstable vector $\rd y_0\in\cC^u_{y_0}$ and denote $\rd y_{k} = D_{y_0}(F_{\bfc_{k-1}}\circ\dots\circ F_{\bfc_0})\,\rd y_0$ for $1\leq k\leq n$. The billiard trajectory from $y_0$ to $y_{n}$ can be broken into at most $2n$ straight legs: Given $1\leq k\leq n$, $y_{k} = F_{\bfc_{k-1}}(y_{k-1})$ equals either $F_{\bfc_{k-1}}^*(F_{\bfc_{k-1}}^*(y_{k-1}))$ or $F_{\bfc_{k-1}}^*(y_{k-1})$ depending, respectively, on whether the trajectory from~$y_{k-1}$ to~$y_k$ involves a collision with the white disk or not. Thus, we obtain collision points $x_j\in\cM^*$, $0\leq j\leq m$, where $n\leq m\leq 2n$. Of course, $(x_j)_{j=0}^m$ has $(y_k)_{k=0}^n$ as an ordered subsequence, with $x_0 = y_0$ and $x_m = y_n$. Likewise, we get a sequence of tangent vectors $\rd x_j$, $0\leq j\leq m$, so that $\rd x_0 = \rd y_0$ and $\rd x_m = \rd y_n$.

The billiard maps $F_{\bfc_k}^*$ of the extended cross section $\cM^*$ satisfy, by well-known \cite{ChernovMarkarian_2006} formulas,
\beqn
\frac{\|\rd x_{j+1}\|_\rp}{\|\rd x_j\|_\rp} 
=  1+\tau_j\frac{\cV_j + \kappa_j}{\cos\varphi_j}. 
\eeqn
Here $\|\slot\|_\rp$ is the so-called p-metric of the tangent space, $\cV_j$ refers to the slope of $\rd x_j$, $\kappa_j$ to the curvature of the boundary at $x_j=(r_j,\varphi_j)$, and $\tau_j$ to the length of the straight leg from~$x_j$ to~$x_{j+1}$. Since $\rd x_j\in \cC^u_{x_j}$, we have
\beqn
\frac{\|\rd x_{j+1}\|_\rp}{\|\rd x_j\|_\rp}
 \geq \Lambda
 \quad\text{with}\quad
\Lambda =  1+\tau_{\min}a_{\min}  > 1.
\eeqn
Since the Euclidean metric is related to the p-metric by the formula
\beqn
\|\rd x\| = \frac{\|\rd x\|_\rp}{\cos\varphi}\sqrt{1+\cV^2},
\eeqn
we have
\beqn
\frac{\|\rd y_n\|}{\|\rd y_0\|} = \frac{\|\rd x_m\|}{\|\rd x_0\|} = \frac{\|\rd x_m\|_\rp}{\|\rd x_1\|_\rp} \frac{\|\rd x_1\|_\rp}{\|\rd x_0\|_\rp} \frac{\cos\varphi_0}{\cos\varphi_m}\frac{\sqrt{1+\cV_m^2}}{\sqrt{1+\cV_0^2}},
\eeqn
where
\beq\label{eq:bounded_contraction}
\frac{\|\rd x_1\|_\rp}{\|\rd x_0\|_\rp} \frac{\cos\varphi_0}{\sqrt{1+\cV_0^2}} = \frac{\cos\varphi_0 + \tau_0(\cV_0+\kappa_0)}{\sqrt{1+\cV_0^2}} \geq \frac{\tau_{\min}a_{\min}}{\sqrt{1+b_{\max}^2}}.
\eeq
Therefore
\beqn
\frac{\|\rd y_n\|}{\|\rd y_0\|} \geq  C\Lambda^m \geq C\Lambda^n
\eeqn
with
\beqn
C = \Lambda^{-1}  \frac{\tau_{\min}a_{\min}}{\sqrt{1+b_{\max}^2}} \sqrt{1+a_{\min}^2},
\eeqn
which proves the lemma.
\end{proof}

\begin{remark}\label{rem:non-uniform}
Notice in the proof that if the image point $x_1$ was on a transparent wall, $\kappa_1= 0$, and if in that case it was possible for $\cos\varphi_1$ to be arbitrarily small in contrast to \eqref{eq:transparent_cosine}, then \eqref{eq:slope_formula} would result in arbitrarily small values of the slope $\cV_1$ of the image tangent vector. Subsequently, we would not obtain a positive lower bound $a_{\min}$ in \eqref{eq:cone_slopes}. Even worse, if it was possible for $\cos\varphi_0$ to be arbitrarily small when $x_0$ is on a transparent wall, we would not obtain a positive lower bound in~\eqref{eq:bounded_contraction}, nor the uniform expansion rate $\Lambda>1$ in Lemma~\ref{lem:hyperbolic}. These facts constitute the second fundamental reason, alluded to in Section~\ref{sec:Model}, why we defined the cross section $\cM$ in the specific way, ruling out small values of the cosine on transparent walls.
\end{remark}

\subsection{Homogeneous local stable manifolds}\label{sec:HLSM}
In order to control distortion effects of dispersing billiard maps, it is traditional to introduce so-called homogeneity strips. In our case, 
\begin{eqnarray*}
\bH_k & = & \{ (r,\varphi)\in\cM^*\,:\, \pi/2-k^{-2}<\varphi \leq \pi/2-(k+1)^{-2}\}\\
\bH_{-k} &= &\{ (r,\varphi)\in\cM^*\,:\, -\pi/2+(k+1)^{-2}\leq \varphi < -\pi/2+k^{-2}\}
\end{eqnarray*}
for all integers $k\geq k_0$, where $k_0$ is a sufficiently large uniform constant.
We also set
\beqn
\bH_0 = \{ (r,\varphi)\in\cM^*\,:\, -\pi/2+k_0^2\leq \varphi \leq \pi/2-k_0^{-2}\}\ .
\eeqn
Without going into any detail, if $F_\bfc^*(x)$ and $F_\bfc^*(y)$ belong to the same homogeneity strip, then roughly speaking the derivatives $D_xF_\bfc^*$ and $D_yF_\bfc^*$ are comparable.
By~\eqref{eq:transparent_cosine}, we may assume that~$k_0$ is so large that each component~$\cM_i$ with~$i$ even (transparent wall) involves just one such strip, namely~$\bH_0$. 

Suppose that $(\bfc_n)_{n\geq 0}$ is a sequence of admissible centerings. We say that two points~$x$ and~$y$ in~$\cM^*$ are separated if they lie either in different connected components of~$\cM^*$ or in different homogeneity strips of the same component. Recall the notation in~\eqref{eq:composition}. Given a pair $(x,y)\in\cM\times\cM$, the future \emph{future separation time} $s_+(x,y)=s_+((\bfc_n)_{n\geq 0};x,y)$ of $x$ and $y$ is defined as follows. If~$x$ and~$y$ are separated, $s_+(x,y)=0$. Otherwise $s_+(x,y)$ is the smallest integer $n\geq 1$ such that either $F_{\bfc_{n-1}}^*(\cF_{n-1}(x))$ and $F_{\bfc_{n-1}}^*(\cF_{n-1}(y))$, or $\cF_n(x)$ and $\cF_n(y)$, are separated. (The first alternative takes into account the possibility that the trajectories land on different homogeneity strips on the \emph{white} disk, or that only one of them hits the white disk, before making a return to $\cM$.)

Let $(\bfc_n)_{n\geq 0}$ be an infinite sequence of admissible centerings. A connected smooth curve $W$ is called a {\it local stable manifold}, if the following hold for every $n \geq 0$: (i) $\cF_n (W)$ is connected and homogeneous, (ii) $T_x(\cF_n (W)) \subset \cC^s_x$ for every $x \in \cF_n (W)$, and (iii) similar statements hold for $F_{\bfc_n}((\cF_n(W)))$. (Item (iii) takes into account the possibility that the curve hits the white scatterer before returning to $\cM$.) The local stable manifolds are absolutely continuous with uniform bounds on the Jacobian, the key factor behind it being the uniform hyperbolicity discussed before. By the same token, they contract at a uniform exponential rate (Lemma~\ref{lem:hyperbolic}). Notice that, by definition, two points on the same homogeneous local stable manifold never separate, so that their future separation time is infinite. It turns out that this is a defining property of homogeneous local stable manifolds. See~\cite{StenlundYoungZhang_2012} for more on local stable manifolds in the time-dependent setting.

\subsection{Exponential loss of memory}\label{sec:memory_loss}
Given any admissible centering~$\bfc$, classical theory of billiards shows that the map $F_\bfc$, corresponding to the \emph{fixed} center~$\bfc$, is \emph{ergodic and mixing} with respect to the invariant measure $\mu$. 

For the compositions \eqref{eq:composition} of such maps with \emph{time-dependent} centerings~$\bfc_n$, the analysis carried out in~\cite{StenlundYoungZhang_2012} is directly applicable. The main result of that paper can be stated in the present context as follows:
\begin{thm}\label{thm:weak_conv}
There exists $\ve>0$ such that the following holds. Let $\mu^1$ be a probability measure on $\cM$, with
a strictly positive, $\tfrac16$-H\"older continuous density~$\rho^1$ 
with respect to the measure $\mu$.
Given $\gamma>0$, there exist $0<\theta_\gamma<1$ and $C_\gamma>0$ such that
\beqn
\left|\int_{\cM} f\circ \cF_n\, \rd\mu^1 - \int_{\cM} f \, \rd\mu \right| \leq C_\gamma ({\| f \|}_\infty + {|f|}_\gamma)\theta_\gamma^n, \quad n\geq 0,
\eeqn
for all sequences $(\bfc_n)_{n=0}^\infty$ of admissible centerings and all
 $\gamma$-H\"older continuous \mbox{$f:\cM\to\bR$}. The constant $C_\gamma=C_\gamma(\rho^1)$ depends on the density $\rho^1$ through the H\"older constant of $\log\rho^1$, while $\theta_\gamma$ does not depend 
on $\mu^1$.
\end{thm}
Noting that $\int_{\cM} f\circ \cF_n\, \rd\mu^1 = \int_{\cM} f\, \rd(\cF_n)_*\mu^1$, the theorem states that the push-forward measure $(\cF_n)_*\mu^1$ converges at a \emph{uniform} exponential rate to $\mu$, in a properly understood weak sense. In particular, recalling that mixing can be formulated as correlation decay (see Theorem~\ref{thm:pair_bound}), when applied to constant sequences $\bfc_n = \bfc$, $n\geq 0$, the above theorem yields a \emph{uniform} exponential mixing rate for H\"older continuous observables, for all fixed admissible maps $F_\bfc$.

Theorem~\ref{thm:weak_conv} has a pre-eminent role in the proof of Theorem~\ref{thm:ASIP}. In truth, we will need a slightly refined version of it (Lemma~\ref{lem:memory_loss}) which allows for the measure $\mu^1$ to be supported on a single ``unstable'' curve. Nevertheless, it seems appropriate to introduce the above formulation as a key ingredient of the theory.


\section{The proof}\label{sec:Proof}
The proof of Theorem~\ref{thm:ASIP} --- as is usually the case with billiard proofs --- is rather technical. In order to keep the paper in manageable proportions, we will need to assume that the reader has a working knowledge of billiard techniques. Nevertheless, the flow of the argument should be accessible to a nonspecialist. For background on billiards, we refer to the excellent textbook~\cite{ChernovMarkarian_2006} and, on the time-dependent theory, to the paper~\cite{StenlundYoungZhang_2012}.

The proof will proceed as follows. In Section~\ref{sec:proof_defs} we introduce some notations and define the spaces of observables. In Section~\ref{sec:proof_pair_corr} we obtain uniform, exponentially decaying, bounds on pair correlation functions for observables in those spaces. In Section~\ref{sec:proof_multiple_corr} we extend the result to multiple correlation functions. Using the multiple correlation bounds, we can control the characteristic function of the process, and thus verify Gou\"ezel's condition for the vector-valued almost sure invariance principle. This is done in Section~\ref{sec:ASIP_proof}. In the non-stationary case, we also need a bound on the convergence rate of the covariance, which is obtained in Section~\ref{sec:proof_covariance}. Finally, Lemma~\ref{lem:coboundary} on the (non-)degeneracy of the covariance is proved in Section~\ref{sec:proof_coboundary}.


\subsection{Some definitions}\label{sec:proof_defs}
Here we will denote by
\beqn
\Omega = (B_\varepsilon(\mathbf{0}))^\bZ
\eeqn
the space of bi-infinite sequences of admissible centerings. Elements of $\Omega$ are denoted by $\omega=(\omega_i)_{i\in\bZ}$, where each $\omega_i$ is thus an admissible centering inside the disk $B_\varepsilon(\mathbf{0})$. The usual left shift is denoted by $\sigma$, i.e., 
\beqn
\sigma:\Omega\to\Omega:(\sigma\omega)_i = \omega_{i+1}\,\,\,\forall\, i\in\bZ.
\eeqn
For technical reasons, we will work with bi-infinite sequences, although the sequences $(F_{\omega_i})_{i\geq 0}$ are what we are eventually interested in.

Given a sequence $\omega\in\Omega$, we will use notation such as
\beqn
\cF_n(\omega,\slot) =F_{\omega_{n-1}}\circ\dots\circ F_{\omega_0} (\slot)
\quad \text{and} \quad 
\cF_{n+m,n+1}(\omega,\slot)
 = \cF_m(\sigma^n\omega,\slot),
\eeqn
as well as
\beqn
\cF_n^{-1}(\omega,\slot) =(F_{\omega_{n-1}}\circ\dots\circ F_{\omega_0})^{-1}(\slot)
\quad \text{and} \quad
\cF_{n+m,n+1}^{-1}(\omega,\slot)=\cF_m^{-1}(\sigma^n\omega,\slot)
\eeqn
for $n,m\geq 0$.
We use the convention that
$
\cF_{n,n+1} = \mathrm{id}_\cM
$
throughout.
The notation is consistent with the convention that, given $\omega$, the maps $F_{\omega_n}$ and $\cF_{n+m,n+1}(\omega,\slot)$ describe the transformations from time $n$ to time $n+1$ and to time $n+m$, respectively. 

Let 
$$
\cP_+(\omega)=(W_+(\omega)_\alpha)_{\alpha\in\cA_+(\omega)}
$$
be the measurable partition of $\cM$ consisting of the homogeneous local \emph{unstable} manifolds corresponding to the past $(\dots,\omega_{-2},\omega_{-1})$ and
$$
\cP_-(\omega)=(W_-(\omega)_\alpha)_{\alpha\in\cA_-(\omega)}
$$
the measurable partition of $\cM$ consisting of the homogeneous local \emph{stable} manifolds corresponding to the future $(\omega_{0},\omega_{1},\dots)$.
Also set
$$
\cP_{\pm}^{(n)}(\omega) = \cP_{\pm}(\sigma^n\omega), \quad n\in\bZ.
$$
We often write $\cP_\pm^{(n)}(\omega) = \bigl(W_{\pm,\alpha}^{(n)}(\omega)\bigr)_{\alpha\in\cA_\pm^{(n)}(\omega)}$, where
\beqn
\cA_\pm^{(n)}(\omega) = \cA_\pm(\sigma^n\omega)\quad\text{and}\quad W_{\pm,\alpha}^{(n)}(\omega) = W_{\pm,\alpha}(\sigma^n\omega),
\eeqn
or simply
\beqn
\cP_\pm^{(n)} = \bigl(W_{\pm,\alpha}^{(n)}\bigr)_{\alpha\in\cA_\pm^{(n)}}
\eeqn
when there is no danger of confusion about the sequence $\omega$. Note that, for any $n\in\bZ$, $\cP_+^{(n+1)}(\omega)$ is a refinement of the partition $F_{\omega_n}(\cP_+^{(n)}(\omega))$, obtained by cutting the elements of the latter into maximal homogeneous components. Similarly, $\cP_-^{(n)}(\omega)$ is a refinement of $F_{\omega_n}^{-1}(\cP_-^{(n+1)}(\omega))$.

Recall that, given a one sided sequence $(\omega_0,\omega_1,\dots)$ and two points $x,y\in\cM$, we have defined the future separation time of~$x$ and~$y$ in Section~\ref{sec:HLSM}. The definition obviously extends to two-sided sequences $\omega\in\Omega$ and depends only on $\omega_i$, $i\geq 0$. We write $s_+(\omega;x,y)$ for the future separation time. Similarly, the past separation time $s_-(\omega;x,y)$ is defined going backward in time, via the \emph{inverse} maps $F_{\omega_{-1}}^{-1},\, F_{\omega_{-2}}^{-1}\circ F_{\omega_{-1}}^{-1},\, \dots$, and it depends only on $\omega_i$, $i<0$.

Given a sequence $\omega$ and a number $\vartheta\in(0,1)$, we say that a function $f:\cM\to\bC$ is dynamically H\"older continuous on the homogeneous local unstable manifolds associated to $\omega$ with rate $\vartheta$, if there exists a real constant $K\geq 0$ such that
\beqn
|f(x)-f(y)|\leq K\vartheta^{s_+(\omega;x,y)}
\eeqn 
holds for all $x,y\in W_{+,\alpha}$ and all $\alpha\in\cA_+$. The smallest possible such constant $K$ is denoted $K_+(\omega;\vartheta,f)$. The class of such functions is a vector space, which we denote $\cH_+(\omega;\vartheta)$.
Similarly, a function $g:\cM\to\bC$ is called dynamically H\"older continuous on the homogeneous local stable manifolds associated to $\omega$ with rate $\vartheta$, if there exists a real constant $K\geq 0$ such that
\beqn
|g(x)-g(y)|\leq K\vartheta^{s_-(\omega;x,y)}
\eeqn 
holds for all $x,y\in W_{-,\alpha}$ and all $\alpha\in\cA_-$. The smallest possible such constant $K$ is denoted $K_-(\omega;\vartheta,f)$. The class of such functions is a vector space, which we denote $\cH_-(\omega;\vartheta)$.

The above function spaces generalize similar spaces of \cite{Stenlund_2010} to the \emph{time-dependent setting}, whereas the spaces of \cite{Stenlund_2010} generalize those of \cite{Chernov_2006}. In particular, ordinary H\"older continuous functions are dynamically H\"older continuous:

\begin{lem}\label{lem:Holder_dyn_Holder}
There exists a uniform constant $\bar C>0$ such that the following holds. Given $\bar\omega_0\in B_\ve(\mathbf{0})$, $\gamma>0$ and $A_f>0$, suppose $f:\cM\to\bC$ is a bounded function that satisfies 
\beqn
|f(x)-f(y)|\leq A_f \,d(x,y)^\gamma
\eeqn
for all $x$ and $y$ belonging to the same homogeneous component of $\cM\setminus \cS_{\bar\omega_0}$. (Here $d(\slot,\slot)$ is the Euclidean distance in $\cM$.) Then $f\in\cH_+(\omega;\Lambda^{-\gamma})\cap\cH_-(\omega;\Lambda^{-\gamma})$ with
\beqn
K_+(\omega;\Lambda^{-\gamma},f) \leq \max(A_f\bar C^\gamma,2\|f\|_\infty)\Lambda^\gamma
\quad\text{and}\quad
K_-(\omega;\Lambda^{-\gamma},f) \leq A_f\bar C^\gamma\Lambda^\gamma,
\eeqn
for all $\omega\in\Omega$ with $\omega_0=\bar \omega_0$.
\end{lem}
\begin{proof}
Let $D$ be the maximal diameter of the homogeneous components of $\cM$.
Suppose $\omega$ satisfies $\omega_0=\bar\omega_0$. For any $x,y\in W_{+,\alpha}$ and $\alpha\in\cA_+$,
$
C\Lambda^n d(x,y) \leq d(\cF_n(\omega,x),\cF_n(\omega,y)) \leq D
$
if $0\leq n< s_+(\omega;x,y)$. This follows from the uniform hyperbolicity (Lemma~\ref{lem:hyperbolic}). We then arrive at
$
d(x,y) \leq \bar C \Lambda\Lambda^{-s_+(\omega;x,y)}
$
with the uniform constant $\bar C = DC^{-1}$. We either have $s_+(\omega;x,y)>1$ or $s_+(\omega;x,y)=1$. In the first case, $x$ and $y$ belong to the same homogeneous component of $\cM\setminus \cS_{\bar\omega_0}$, so the assumed H\"older continuity of $f$ yields
$
|f(x)-f(y)| \leq A_f(\bar C \Lambda\Lambda^{-s_+(\omega;x,y)})^\gamma.
$
In the second case, $|f(x)-f(y)|\leq 2\|f\|_\infty = 2\|f\|_\infty\Lambda^{\gamma}(\Lambda^{-\gamma})^{s_+(\omega;x,y)}$. The claim concerning $\cH_+(\omega;\Lambda^{-\gamma})$ follows. Assume now that $x,y\in W_{-,\alpha}$ and $\alpha\in\cA_-$, instead. By similar reasoning,
$
d(x,y) \leq \bar C\Lambda \Lambda^{-s_-(\omega;x,y)}.
$
Forward images of homogeneous local stable manifold do not, by construction, meet singularity curves, so that $W_{-,\alpha}$ belongs to a single homogeneous component of $\cM\setminus\cS_{\bar\omega_0}$. Thus, $
|f(x)-f(y)| \leq A_f(\bar C \Lambda\Lambda^{-s_-(\omega;x,y)})^\gamma.
$ 
\end{proof}

\begin{remark}\label{rem:fixed_dyn_Holder}
All of the above applies when the billiard map does not depend on time (fixed scatterers). In that case the partitions~$\cP_\pm$ are defined by the given scatterer configuration instead of a sequence, and the same is true of the separation times~$s_\pm$ as well as the function classes~$\cH_\pm$.
\end{remark}

\subsection{Pair correlation bounds}\label{sec:proof_pair_corr}
We say that $\omega\in\Omega$ has a saturating past, if there exists an $i_0<0$ such that $\omega_i=\omega_{i_0}$ for all $i\leq i_0$. Denote by $\Omega_\mathrm{sat}\subset\Omega$ the set of all sequences in $\Omega$ with a saturating past. Notice that 
\beqn
\sigma(\Omega_\mathrm{sat})\subset \Omega_\mathrm{sat},
\eeqn
that is, the left shift of a sequence with a saturating past has a saturating past. The class $\Omega_\mathrm{sat}$ will be needed in place of $\Omega$ because of Lemma~\ref{lem:proper_tail}; see the remark after it.

The next result gives an exponential upper bound on pair correlation functions for observables of the type introduced above.

\begin{thm}\label{thm:pair_bound}
For a sufficiently small $\ve>0$, there exist uniform constants $C_\mathrm{p}>0$, $C>0$, $\zeta>0$, and $\chi>0$ such that, for every $\vartheta_\pm\in(0,1)$, every $\omega\in\Omega_\mathrm{sat}$, every $n\geq 0$, and every pair $(f,g)\in\cH_+(\omega;\vartheta_+)\times \cH_-(\sigma^n\omega;\vartheta_-)$,
\beqn
\begin{split}
& 
\left|\int f\cdot g\circ\cF_{n}(\omega,\slot)\,\rd\mu-\int f\,\rd\mu\int g\,\rd\mu\right|
\\
& \qquad 
\leq \| f\|_\infty B_-(\sigma^n\omega;\vartheta_-,g)\,(\theta(\vartheta_-))^{n/4} +  2C_\mathrm{p}\|f\|_\infty\|g\|_\infty\, e^{-n/4\chi}+ 2K_+(\omega;\vartheta_+,f)\|g\|_\infty\, \vartheta_+^{n/2-1}.
\end{split}
\eeqn
Here $B_-(\sigma^n\omega;\vartheta_-,g)=C\|g\|_\infty + K_-(\sigma^n\omega;\vartheta_-,g)\,\vartheta_-^{-1}$ and $\theta(\vartheta_-) = \max(\zeta,\vartheta_-^{1/2})\in(0,1)$.
\end{thm}

Recalling Lemma~\ref{lem:Holder_dyn_Holder}, we point out that the previous theorem applies to the special case of ``ordinary'' H\"older continuous observables and yields an exponential rate of pair correlation decay for them.

As we will see later on, Theorem~\ref{thm:pair_bound} is surprisingly strong in that it implies a bound on \emph{multiple} correlation functions, which is seemingly much more general a result. To understand how this is possible, one has to appreciate two aspects of Theorem~\ref{thm:pair_bound}. First, it involves carefully designed and sufficiently large classes of observables. Second, the bound is completely explicit in regard to its dependence on the observables $f$ and $g$. In conjunction with Lemma~\ref{lem:products_regular} below, Theorem~\ref{thm:pair_bound} thus becomes a very useful tool for proving limit results.

\begin{proof}[Proof of Theorem~\ref{thm:pair_bound}]
Denoting
\beqn
\bar g=g-\int g\,\rd\mu,
\eeqn
we have
\beqn
\begin{split}
\int f\cdot g\circ \cF_{n+m}(\omega,\slot)\,\rd\mu-\int f\,\rd\mu\int g\,\rd\mu
= \int f\circ\cF_n^{-1}(\omega,\slot)\cdot \bar g\circ\cF_{n+m,n+1}(\omega,\slot)\,\rd\mu
\end{split}
\eeqn
for any fixed $\omega$. As $\omega$ is fixed for the rest of the proof, we will often omit it from the notation. Using the partition~$\cP_+^{(n)}$, the measure~$\mu$ disintegrates into a probability measure $\lambda^{(n)}$ on $\cA_+^{(n)}$ and a system of conditional probability measures $\nu_\alpha^{(n)}$ on the partition elements $W_{+,\alpha}^{(n)}$ such that
\beqn
\int h\,\rd\mu = \int_{\cA_+^{(n)}}\int_{W_{+,\alpha}^{(n)}} h \,\rd\nu_\alpha^{(n)}\,\rd\lambda^{(n)}(\alpha)
\eeqn
holds for any Borel measurable function $h:\cM\to\bC$.
We can thus write the above identity as
\beqn
\int f\cdot g\circ \cF_{n+m}\,\rd\mu-\int f\,\rd\mu\int g\,\rd\mu = I_1 + I_2
\eeqn
where
\beqn
I_1 = \int_{\cA_+^{(n)}}\int_{W_{+,\alpha}^{(n)}} \left[ f\circ\cF_n^{-1} - \int_{W_{+,\alpha}^{(n)}} f\circ\cF_n^{-1}\,\rd\nu_\alpha^{(n)}\right]\cdot \bar g\circ\cF_{n+m,n+1}\,\rd\nu_\alpha^{(n)}\,\rd\lambda^{(n)}(\alpha)
\eeqn
and
\beqn
I_2 = \int_{\cA_+^{(n)}} \left[ \int_{W_{+,\alpha}^{(n)}} f\circ\cF_n^{-1}\,\rd\nu_\alpha^{(n)}\right]\cdot \left[\int_{W_{+,\alpha}^{(n)}} \bar g\circ\cF_{n+m,n+1}\,\rd\nu_\alpha^{(n)}\right]\rd\lambda^{(n)}(\alpha).
\eeqn

Let us bound $I_1$ first. Because each $\nu_\alpha^{(n)}$ is a probability measure,
\beqn
\begin{split}
&
|I_1|
\leq \sup_{\cA_+^{(n)}}\sup_{x,y\in W_{+,\alpha}^{(n)}} \left( f\circ\cF_n^{-1}(x) - f\circ\cF_n^{-1}(y)\right)\cdot \int |\bar g|\,\rd\mu.
\end{split}
\eeqn
For any $\alpha\in\cA_+^{(n)}$ there exists a $\beta\in\cA_+$ such that for any pair of points $x,y\in W_{+,\alpha}^{(n)}$ we have $\cF_n^{-1}(x),\cF_n^{-1}(y)\in W_{+,\beta}$. Therefore,
\beqn
 s_+\bigl(\omega;\cF_n^{-1}(\omega,x),\cF_n^{-1}(\omega,y)\bigr) \geq n.
\eeqn
Since $f\in\cH_+(\omega;\vartheta_+)$, we obtain the bound
\beqn
|f\circ\cF_n^{-1}(x) - f\circ\cF_n^{-1}(y)|\leq K_+(\omega;\vartheta_+,f)\vartheta^n,
\eeqn
for all $x,y\in W_{+,\alpha}$, for all $\alpha\in\cA_+$. Now $\sup|\bar g|\leq 2\sup|g|$ yields
\beq\label{eq:I_1_last_bound}
|I_1|\leq  2\sup|g|\,K(\omega;\vartheta_+,f)\vartheta_+^n.
\eeq

Coming to $I_2$, notice that
\beq\label{eq:I_2_first_bound}
\begin{split}
|I_2| \leq \sup | f| \cdot \int_{\cA_+^{(n)}}  \left|\int_{W_{+,\alpha}^{(n)}} \bar g\circ\cF_{n+m,n+1}\,\rd\nu_\alpha^{(n)}\right| \rd\lambda^{(n)}(\alpha).
\end{split}
\eeq
The following result is at the heart of bounding $I_2$. Recall that $\ve$ is the upper bound on the displacement of the center of the white disk from the center of the unit square. 

\begin{lem}\label{lem:memory_loss}
For a sufficiently small $\ve>0$, there exist uniform constants $C>0$, $\chi>0$, and $\zeta>0$ such that, for every $\omega$, every $\vartheta_-\in(0,1)$, and every $g\in\cH_-(\sigma^{n+m}\omega;\vartheta_-)$,
\beq\label{eq:memory_loss}
\begin{split}
\left|\int_{W_{+,\alpha}^{(n)}} \bar g\circ\cF_{n+m,n+1}\,\rd\nu_\alpha^{(n)}\right| \leq B_-(\sigma^{n+m}\omega;\vartheta_-,g)\,(\theta(\vartheta_-))^{m-\chi\left|\log|W_{+,\alpha}^{(n)}| \right |}
\end{split}
\eeq
for all $\alpha\in\cA^{n}(\omega)$ and all $n,m\geq 0$.  Here $B_-(\sigma^{n+m}\omega;\vartheta_-,g)=C\|g\|_\infty + K_-(\sigma^{n+m}\omega;\vartheta_-,g)\,\vartheta_-^{-1}$ and $\theta(\vartheta_-) = \max(\zeta,\vartheta_-^{1/2})\in(0,1)$.
\end{lem}

We remark that the above lemma is a statement of statistical memory loss starting from an initial measure supported on a single unstable curve. This is the strong version of Theorem~\ref{thm:weak_conv} hinted at in the last paragraph of Section~\ref{sec:memory_loss}. 
\begin{proof}[Proof of Lemma~\ref{lem:memory_loss}]
This follows directly from the time-dependent coupling argument in \cite{StenlundYoungZhang_2012} as we now explain. In \cite{StenlundYoungZhang_2012} we proved the analogous statement for ``ordinary'' H\"older continuous observables~$g$. Here we just need to check that the estimate one obtains for~$g\in\cH_-(\sigma^{n+m}\omega;\vartheta_-)$ is the one appearing in~\eqref{eq:memory_loss}.
Because
\beqn
\begin{split}
\int_{W_{+,\alpha}^{(n)}} \bar g\circ\cF_{n+m,n+1}\,\rd\nu_\alpha^{(n)} = \int_{W_{+,\alpha}^{(n)}} g\circ\cF_{n+m,n+1}\,\rd\nu_\alpha^{(n)} - \int_\cM g\circ\cF_{n+m,n+1}\,\rd\mu,
\end{split}
\eeqn
the desired bound is related to the convergence of the push-forward $(\cF_{n+m,n+1})_*\nu_\alpha^{(n)}$ to the (invariant) measure $\mu$ with increasing~$m$ in a weak sense. It is shown in~\cite{StenlundYoungZhang_2012} that there exists a uniform constant $\chi>0$ such that, for $m'= \bigl[\chi\bigl |\log|W_{+,\alpha}^{(n)}|\bigr|\bigr]$, the measure $(\cF_{n+m',n+1})_*\nu_\alpha^{(n)}$ is ``proper'' in the sense of~\cite{StenlundYoungZhang_2012}. (The terminology originated in~\cite{Chernov_2006} and the references therein.) Moreover, the measure $\mu$ is readily proper. 
It is therefore certainly enough to show that 
\beq\label{eq:mem_loss_temp}
\left|\int g \circ \cF_{n+m,n+m'+1} \,\rd\mu^1 - \int g \circ \cF_{n+m,n+m'+1} \,\rd\mu^2 \right|\leq B_-(\sigma^{n+m}\omega;\vartheta_-,g)\,(\theta(\vartheta_-))^{m-m'}
\eeq
for any two proper measures $\mu^1$ and $\mu^2$.

We recall from \cite{StenlundYoungZhang_2012} that, for each $m\geq m'+1$, two proper measures $\mu^i$ ($i=1,2$) can be decomposed as $\mu^i_{\lceil (m-m')/2 \rceil} + \sum_{j=0}^{\lceil (m-m')/2 \rceil}\bar\mu^i_j$. Here $\mu^i_{\lceil (m-m')/2 \rceil}$ are the measures that remain \emph{uncoupled} after $\lceil (m-m')/2 \rceil$ steps (counting from time $n+m'$). Provided $\ve>0$ is small enough, we can assume that they satisfy $\mu^1_{\lceil (m-m')/2 \rceil}(\cM) = \mu^2_{\lceil (m-m')/2 \rceil}(\cM)\leq C\zeta^{m-m'}$ for uniform constants $C>0$ and $\zeta\in(0,1)$. Therefore,
\beqn
\begin{split}
&\left|\int g \circ \cF_{n+m,n+m'+1} \,\rd\mu^1 - \int g \circ \cF_{n+m,n+m'+1} \,\rd\mu^2 \right|
\\
&\qquad\qquad \leq   C\|g\|_\infty \zeta^{m-m'} + 
\sum_{j=0}^{\lceil (m-m')/2 \rceil} \left|\int g \circ \cF_{n+m,n+m'+1} \,\rd\bar\mu^1_j - \int g \circ \cF_{n+m,n+m'+1} \,\rd\bar\mu^2_j \right|.
\end{split}
\eeqn
On the other hand, the measures $(\cF_{n+m'+j,n+m'+1})_*\bar\mu_j^i$ ($i=1,2$) are \emph{coupled}: Both measures are supported on a family (called a ``magnet'') of homogeneous local stable manifolds $W^s_{n+m'+j}$, corresponding to the future $(\omega_{n+m'+j},\omega_{n+m'+j+1}\dots)$, in such a way that their masses on each such local stable manifold agree. In particular $\bar\mu_j^1(\cM) = \bar\mu_j^2(\cM)$. Notice that, for each pair of points $x,y$ in the same local stable manifold~$W^s_{n+m'+j}$, 
\beqn
\begin{split}
& |g\circ  \cF_{n+m,n+m'+j+1}(x)-g\circ \cF_{n+m,n+m'+j+1}(y)|
\\
& \qquad\qquad
\leq K_-(\sigma^{n+m}\omega;\vartheta_-,g)\,\vartheta_-^{s_-^{n+m}(\cF_{n+m,n+m'+j+1}(x),\cF_{n+m,n+m'+j+1}(y))}
\\
& \qquad\qquad
\leq K_-(\sigma^{n+m}\omega;\vartheta_-,g)\,\vartheta_-^{m-m'-j + s_-^{n+m'+j}(x,y)} \leq K_-(\sigma^{n+m}\omega;\vartheta_-,g)\,\vartheta_-^{m-m'-j},
\end{split}
\eeqn
because $g\in\cH_-(\sigma^{n+m}\omega;\vartheta_-)$. Since $\sum_{j}\bar\mu_j^i(\cM)=1$, it then follows that 
\beqn
\begin{split}
&\sum_{j=0}^{\lceil (m-m')/2 \rceil}\left|\int g \circ \cF_{n+m,n+m'+1} \,\rd\bar\mu^1_j - \int g \circ \cF_{n+m,n+m'+1} \,\rd\bar\mu^2_j \right|
\\
&\qquad\qquad
\leq K_-(\sigma^{n+m}\omega;\vartheta_-,g) \sum_{j=0}^{\lceil (m-m')/2 \rceil}\bar\mu_j^1(\cM)\,\vartheta_-^{m-m'-j}
\leq K_-(\sigma^{n+m}\omega;\vartheta_-,g)\,\vartheta_-^{-1}(\vartheta_-^{1/2})^{m-m'}.
\end{split}
\eeqn
We refer the interested reader to \cite{StenlundYoungZhang_2012} for a detailed construction of the coupling used above. Combining the obtained estimates yields~\eqref{eq:mem_loss_temp}, which was to be shown.
\end{proof}

In order for Lemma~\ref{lem:memory_loss} to yield a useful bound on the inner integral in~\eqref{eq:I_2_first_bound}, we need to assume $m\gg \chi\bigl|\log|W_{+,\alpha}^{(n)}|\bigr|$ in the exponent appearing in~\eqref{eq:memory_loss}. Thus, for a fixed value of $m$, the partition element $W_{+,\alpha}^{(n)}$ should not be too short. The following lemma provides a tail estimate on the prevalence of short partition elements.

\begin{lem}\label{lem:proper_tail}
There exists a uniform constant $C_\mathrm{p}>0$ such that the following holds. If $\omega\in\Omega_\mathrm{sat}$, then the factor measure $\lambda^{(n)}$ of the measure $\mu$ relative to the partition $\cP_+^{(n)}(\omega)$ satisfies
\beq\label{eq:proper_tail}
\lambda^{(n)}\bigl\{\alpha\in\cA_+^{(n)}\,:\,|W_{+,\alpha}^{(n)}|< \ell\bigr\}\leq C_\mathrm{p}\ell
\eeq
for all $\ell>0$ and for all $n\geq 0$. 
\end{lem}
It is possible to prove that, for small $\ve$,~\eqref{eq:proper_tail} holds for all $n\in\bZ$ and all $\omega\in\Omega\supsetneq\Omega_\mathrm{sat}$ (starting from a generalization of Theorem~5.17 of~\cite{ChernovMarkarian_2006}). However, since the proof is long and technical, we resort to the weaker assertion given in Lemma~\ref{lem:proper_tail}, which is perfectly sufficient for our needs.
\begin{proof}[Proof of Lemma~\ref{lem:proper_tail}]
Denoting
$
\cZ^{(n)} = \int_{\cA_+^n} |W_{+,\alpha}^{(n)}|^{-1}  \,\rd\lambda^{(n)}(\alpha)
$
for all $n\in\bZ$, it can be shown using the uniform growth lemma in the time-dependent setting (see \cite{StenlundYoungZhang_2012}) that there exist uniform constants $C_\mathrm{p}>0$ and $\vartheta_\mathrm{p}\in(0,1)$ such that $\cZ^{(n+i)}\leq \tfrac12 C_\mathrm{p}(1+\vartheta_\mathrm{p}^n \cZ^{(i)})$ holds for all $i\in\bZ$, all $n\geq 0$, and all sequences $\omega\in\Omega$. Under the saturating past condition ($\omega\in\Omega_\mathrm{sat}$), $\omega_i=\omega_{i_0}$ holds for all $i\leq i_0$, for some~$i_0$, and it is an exercise in the theory of billiards~\cite{ChernovMarkarian_2006} that $\cZ^{(i_0)}<\infty$ (because the relevant partition $\cP_+^{(i_0)}$ corresponds to the homogeneous unstable manifolds of the \emph{fixed} map $F_{\omega_{i_0}}$). Observing that $\cP_+^{(i)}=\cP_+^{(i_0)}$ for all $i\leq i_0$, we thus obtain $\cZ^{(i_0+m)}\leq \tfrac12 C_\mathrm{p}(1+\vartheta_\mathrm{p}^{i_0+m-i} \cZ^{(i_0)})$ for all $i\leq i_0$ and all $m\geq 0$. Taking the limit $i\to-\infty$, we see in particular that $\cZ^{(n)}\leq \tfrac12 C_\mathrm{p}$ for all $n\geq 0$. Markov's inequality now yields the result: $\lambda^{(n)}\{\alpha\in\cA_+^{(n)}\,:\,|W_{+,\alpha}^{(n)}|< \ell\} \leq \cZ^{(n)}\ell \leq C_\mathrm{p}\ell$.
\end{proof}

We are in position to finish the bound~\eqref{eq:I_2_first_bound} on $I_2$. For any $\ell\in(0,1]$, the decomposition
\beqn
\cA_+^{(n)} = \{\alpha\in\cA_+^{(n)}\,:\,|W_{+,\alpha}^{(n)}|\geq \ell\} \cup \{\alpha\in\cA_+^{(n)}\,:\,|W_{+,\alpha}^{(n)}|< \ell\},
\eeqn
in conjunction with Lemmas~\ref{lem:memory_loss} and~\ref{lem:proper_tail}, yields
\beqn
\begin{split}
|I_2| & \leq \sup | f| \cdot B_-(\sigma^{n+m}\omega;\vartheta_-,g)\,(\theta(\vartheta_-))^{m-\chi\left|\log\ell \right |} \cdot \lambda^{(n)}\{\alpha\in\cA_+^{(n)}\,:\,|W_{+,\alpha}^{(n)}|\geq \ell\}
\\
& \qquad 
+ \sup | f|\cdot \sup|\bar g| \cdot \lambda^{(n)}\{\alpha\in\cA_+^{(n)}\,:\,|W_{+,\alpha}^{(n)}|< \ell\}
\\
& \leq 
 \sup | f| \cdot B_-(\sigma^{n+m}\omega;\vartheta_-,g)\,(\theta(\vartheta_-))^{m-\chi\left|\log\ell \right |} +  \sup |f|\cdot\sup|\bar g|\cdot C_\mathrm{p}\ell
\end{split}
\eeqn
For any $m\geq 0$, we set $\ell=\ell(m)=e^{-m/2\chi}$, so that $\chi|\log\ell|=m/2$. Since $\sup|\bar g|\leq 2\sup|g|$, we get the final bound
\beq\label{eq:I_2_last_bound}
\begin{split}
|I_2| & \leq \sup | f|\cdot B_-(\sigma^{n+m}\omega;\vartheta_-,g)\,(\theta(\vartheta_-))^{m/2} +  2C_\mathrm{p}\sup |f|\sup|g| \,e^{-m/2\chi}.
\end{split}
\eeq

Collecting~\eqref{eq:I_1_last_bound} and~\eqref{eq:I_2_last_bound}, we arrive at
\beqn
\begin{split}
& 
\left|\int f\cdot g\circ\cF_{n+m}\,\rd\mu-\int f\,\rd\mu\int g\,\rd\mu\right|
\\
& \qquad 
\leq \| f\|_\infty B_-(\sigma^{n+m}\omega;\vartheta_-,g)\,(\theta(\vartheta_-))^{m/2} +  2C_\mathrm{p}\,\|f\|_\infty\|g\|_\infty\, e^{-m/2\chi}+ 2\,\|g\|_\infty K_+(\omega;\vartheta_+,f)\,\vartheta_+^n.
\end{split}
\eeqn
Considering separately the two cases $m=n$ and $m=n+1$, and replacing $n+m$ by $n$ in the end, the bound claimed by the theorem follows. This finishes the proof of Theorem~\ref{thm:pair_bound}.
\end{proof}


\subsection{Multiple correlation bounds}\label{sec:proof_multiple_corr} 

As mentioned already, the function spaces introduced above have a special structure. The following lemmas reveal an important facet of that structure. It will serve as the stepping stone from pair correlation bounds to multiple correlation bounds, and eventually to the invariance principle.

\begin{lem}\label{lem:H_shift}
Let $\omega\in\Omega$, $\vartheta\in(0,1)$, and $n\geq 1$ be fixed. For all $f\in\cH_+(\omega;\vartheta)$ and $g\in \cH_-(\sigma^n\omega;\vartheta)$, we have
\beqn
f\circ\cF_n^{-1}(\omega,\slot) \in \cH_+(\sigma^n\omega;\vartheta)
\quad\text{and}\quad
g\circ\cF_n(\omega,\slot)\in \cH_-(\omega;\vartheta)
\eeqn
with
\beqn
K_+(\sigma^n\omega;\vartheta,f\circ\cF_n^{-1}(\omega,\slot))\leq K_+(\omega;\vartheta,f) \,\vartheta^n
\quad\text{and}\quad
K_-(\omega;\vartheta,g\circ\cF_n(\omega,\slot))\leq K_-(\sigma^n\omega;\vartheta,g) \,\vartheta^n.
\eeqn
\end{lem}
\begin{proof}
Suppose $x,y\in W_{-,\alpha}$ for some $\alpha\in\cA_-$. Then $\cF_n(\omega,x),\cF_n(\omega,y)\in W_{-,\beta}^{(n)}$ for some $\beta\in\cA_-^{(n)}$, which in particular means that $s_-(\sigma^n\omega;\cF_n(\omega,x),\cF_n(\omega,y)) = n + s_-(\omega;x,y)$. In other words, for $g\in\cH_-(\sigma^n\omega;\vartheta)$,
\beqn
|g(\cF_n(\omega,x)) - g(\cF_n(\omega,y))| \leq K_-(\sigma^n\omega;\vartheta,g) \,\vartheta^n \,\vartheta^{s_-(\omega;x,y)}.
\eeqn
Since $x,y$ and $\alpha$ were arbitrary, the claim for $g\circ\cF_n(\omega,\slot)$ follows. 

The proof for $f\circ\cF_n^{-1}(\omega,\slot)$ is very similar. Suppose $x,y\in W_{+,\alpha}^{(n)}$ for some $\alpha\in\cA_+^{(n)}$. Then $\cF_n^{-1}(\omega,x),\cF_n^{-1}(\omega,y)\in W_{+,\beta}$ for some $\beta\in\cA_+$, which implies $s_+(\omega;\cF_n^{-1}(\omega,x),\cF_n^{-1}(\omega,y)) = n + s_+(\sigma^n\omega;x,y)$. Thus, for $f\in\cH_+(\omega;\vartheta)$,
\beqn
|f(\cF_n^{-1}(\omega,x)) - f(\cF_n^{-1}(\omega,y))| \leq K_+(\omega;\vartheta,f) \,\vartheta^n \,\vartheta^{s_+(\sigma^n\omega;x,y)}.
\eeqn
The lemma now follows.
\end{proof}
As a corollary, we get the following result.
\begin{lem}\label{lem:products_regular}
Let $\omega\in\Omega$, $n\geq 0$, and $\vartheta_0,\dots\vartheta_n\in(0,1)$ be fixed. We set $\vartheta = \max_{0\leq i\leq n}\vartheta_i$. If $f_i\in\cH_+(\sigma^i\omega;\vartheta_i)$ and $\|f_i\|_\infty<\infty$ for $0\leq i\leq n$, then the function
\beqn
f = \prod_{i=0}^n f_i\circ \cF^{-1}_{n,{i+1}}(\omega,\slot)
\eeqn
belongs to $\cH_+(\sigma^n\omega;\vartheta)$ with
\beqn
K_+(\sigma^n\omega;\vartheta,f) \leq  \sum_{i=0}^n \left(\prod_{j\neq i}\|f_j\|_\infty \right) K_+(\sigma^i\omega;\vartheta_i,f_i)\,\vartheta_i^{n-i}.
\eeqn
Similarly, if $g_i\in\cH_-(\sigma^i\omega;\vartheta_i)$ and $\|g_i\|_\infty<\infty$ for $0\leq i\leq n$, then the function
\beqn
g = \prod_{i=0}^n g_i\circ\cF_i(\omega,\slot)
\eeqn
belongs to $\cH_-(\omega;\vartheta)$ with
\beqn
K_-(\omega;\vartheta,g) = \sum_{i=0}^n \left(\prod_{j\neq i}\|g_j\|_\infty \right) K_-(\sigma^i\omega;\vartheta_i,g_i)\,\vartheta_i^{i}
\eeqn
\end{lem}
\begin{proof}
Consider some $\alpha\in\cA_+^{(n)}$ and two points $x,y\in W_{+,\alpha}^{(n)}$. Notice that $f(x)-f(y)$ equals
\beqn
\begin{split}
\sum_{i=0}^n \left(\prod_{j=0}^{i-1}f_j\bigl(\cF_{n,j+1}^{-1}(\omega,x)\bigr)\right)
 \left[ f_i\bigl(\cF_{n,i+1}^{-1}(\omega,x)\bigr)- f_i\bigl(\cF_{n,i+1}^{-1}(\omega,y)\bigr) \right]
 \left(\prod_{j=i+1}^n f_j\bigl(\cF_{n,j+1}^{-1}(\omega,y)\bigr)\right).
\end{split}
\eeqn
Because $\cF^{-1}_{n,i+1}(\omega,\slot) = \cF_{n-i}^{-1}(\sigma^i\omega,\slot)$, Lemma~\ref{lem:H_shift} yields $f_i\circ\cF_{n,i+1}^{-1}(\omega,\slot)\in \cH_+(\sigma^n\omega;\vartheta_i)$ with
\beqn
K_+(\sigma^n\omega;\vartheta_i,f_i\circ\cF_{n,i+1}^{-1}(\omega,\slot))\leq K_+(\sigma^i\omega;\vartheta_i,f_i) \,\vartheta_i^{n-i}.
\eeqn
Therefore,
\beqn
\begin{split}
|f(x)-f(y)| 
&\leq \sum_{i=0}^n \left(\prod_{j\neq i}\|f_j\|_\infty \right) K_+(\sigma^i\omega;\vartheta_i,f_i) \,\vartheta_i^{n-i} \, \vartheta_i^{s_+(\sigma^n\omega;x,y)}.
\end{split}
\eeqn
The proof of the other part is similar, and we omit it.
\end{proof}

We are now ready to turn to bounding multiple correlation functions. The following theorem gives a rather general and explicit bound. We record it for completeness, as it is interesting in its own right. A special case (Corollary~\ref{cor:multiple_bound}) admitting a much tidier expression for the upper bound will be sufficient for the purpose of proving the main result of the paper.
\begin{thm}\label{thm:multiple_bound}
For a sufficiently small $\ve>0$, there exist uniform constants $C_\mathrm{p}>0$, $C>0$, $\zeta>0$, and $\chi>0$ such that the following holds.
Let $\omega\in\Omega_\mathrm{sat}$, $n\geq 0$, $m\geq 0$, $k\geq 0$, $\vartheta_{+,0}\dots,\vartheta_{+,n}\in (0,1)$ and $\vartheta_{-,0}\dots,\vartheta_{-,m}\in (0,1)$ all be fixed. We set $\vartheta_+ = \max_{0\leq i\leq n}\vartheta_{+,i}$ and $\vartheta_- = \max_{0\leq i\leq m}\vartheta_{-,i}$. Assume that $f_i\in\cH_+(\sigma^i\omega;\vartheta_{+,i})$ and $\|f_i\|_\infty<\infty$ for $0\leq i\leq n$, as well as $g_{n+k+i}\in\cH_-(\sigma^{n+k+i}\omega;\vartheta_{-,i})$ and $\|g_{n+k+i}\|_\infty<\infty$ for $0\leq i\leq m$.
Denote
\beq\label{eq:FG_def}
F = \prod_{i=0}^n f_i\circ\cF_i(\omega,\slot)
\quad\text{and}\quad
G = \prod_{i=0}^m g_{n+k+i}\circ\cF_{n+k+i}(\omega,\slot).
\eeq
Then
\beqn
\begin{split}
&\left|\int FG\,\rd\mu-\int F\,\rd\mu\int G\,\rd\mu\right| 
\leq
\\
& \qquad \left\{C\prod_{i=0}^m\|g_{n+k+i}\|_\infty + \left( \sum_{i=0}^n \left(\prod_{j\neq i}\|g_{n+k+j}\|_\infty \right) K_-(\sigma^{n+k+i}\omega;\vartheta_{-,i}\,,g_{n+k+i})\,\vartheta_{-,i}^{i}\right)\,\vartheta_-^{-1} \right\}
\\
&\qquad\qquad\quad\cdot\left(\prod_{i=0}^n\|f_i\|_\infty\right) \left(\max(\zeta,\vartheta_-^{1/2})\right)^{k/4} 
%
%
\quad +  \quad 2C_\mathrm{p}\left(\prod_{i=0}^n\|f_i\|_\infty\right) \left(\prod_{i=0}^m\|g_{n+k+i}\|_\infty\right) e^{-k/4\chi}
\\
& \qquad+ 2\left(\sum_{i=0}^n \left(\prod_{j\neq i}\|f_j\|_\infty \right) K_+(\sigma^i\omega;\vartheta_{+,i}\,,f_i)\,\vartheta_{+,i}^{n-i}\right)\left(\prod_{i=0}^m\|g_{n+k+i}\|_\infty\right) \vartheta_+^{k/2-1}.
\end{split}
\eeqn
\end{thm}

\begin{proof}
Notice the simple identity
\beqn
\cF_i(\omega,\cF_n^{-1}(\omega,\slot)) = \cF^{-1}_{n,i+1}(\omega,\slot), \quad 0\leq i\leq n.
\eeqn
Let us then define the functions
\beqn
f \equiv F\circ\cF_n^{-1}(\omega,\slot) =  \prod_{i=0}^n f_i\circ\cF_{n,i+1}^{-1}(\omega,\slot)
\eeqn
and
\beqn
g \equiv G\circ\cF_{n+k}^{-1}(\omega,\slot) = \prod_{i=0}^m g_{n+k+i}\circ\cF_{n+k+i,n+k+1}(\omega,\slot).
\eeqn
By invariance of the measure $\mu$,
\beqn
\begin{split}
\int F G\,\rd\mu 
& = \int F(\cF_n^{-1}(\omega,x)) \,G(\cF_n^{-1}(\omega,x)) \,\rd\mu(x)
\\
& = \int F(\cF_n^{-1}(\omega,x)) \,G(\cF_{n+k}^{-1}(\omega,\cF_{n+k,n+1}(\omega,x))) \,\rd\mu(x)
\\
& 
= \int f(x) \,g(\cF_{n+k,n+1}(\omega,x)) \,\rd\mu(x).
\end{split}
\eeqn
For the same reason,
\beqn
\int F \,\rd\mu = \int f \,\rd\mu
\quad\text{and}\quad
\int G \,\rd\mu = \int g \,\rd\mu,
\eeqn
so that
\beqn
\int F G\,\rd\mu - \int F\,\rd\mu\int G\,\rd\mu = \int f \cdot g\circ \cF_{n+k,n+1}(\omega,\slot)\,\rd\mu - \int f\,\rd\mu\int g\,\rd\mu.
\eeqn

Observing that $\cF_{n+k,n+1}(\omega;\slot) = \cF_{k}(\sigma^n\omega;\slot)$, we wish to apply Theorem~\ref{thm:pair_bound} next, and thus bound the right side of the above expression. To do so, we need to prove that $f\in\cH_+(\sigma^n\omega;\vartheta_+)$ and $g\in\cH_-(\sigma^k(\sigma^n\omega);\vartheta_-)$. By assumption, $f_i\in\cH_+(\sigma^i\omega;\vartheta_{+,i})$ for $0\leq i\leq n$, and the first claim follows immediately from Lemma~\ref{lem:products_regular} with
\beq\label{eq:K_+_temp}
K_+(\sigma^n\omega;\vartheta_+,f) \leq  \sum_{i=0}^n \left(\prod_{j\neq i}\|f_j\|_\infty \right) K_+(\sigma^i\omega;\vartheta_{+,i}\,,f_i)\,\vartheta_{+,i}^{n-i}.
\eeq
On the other hand, $g = \prod_{i=0}^m g_{n+k+i}\circ\cF_{i}(\sigma^{n+k}\omega,\slot)$. By the assumption of the theorem, $g_{n+k+i}\in\cH_-(\sigma^i(\sigma^{n+k}\omega);\vartheta_{-,i})$ for $0\leq i\leq m$. Thus, according to Lemma~\ref{lem:products_regular}, the function~$g$ is indeed in $\cH_-(\sigma^{k}(\sigma^n\omega);\vartheta_-)$ with
\beq\label{eq:K_-_temp}
K_-(\sigma^{k}(\sigma^n\omega);\vartheta_-,g) \leq \sum_{i=0}^n \left(\prod_{j\neq i}\|g_{n+k+j}\|_\infty \right) K_-(\sigma^{n+k+i}\omega;\vartheta_{-,i}\,,g_{n+k+i})\,\vartheta_{-,i}^{i}.
\eeq
Now, Theorem~\ref{thm:pair_bound} yields the bound
\beqn
\begin{split}
& \left|\int FG\,\rd\mu-\int F\,\rd\mu\int G\,\rd\mu\right| = \left|\int f \cdot g\circ \cF_{k}(\sigma^n\omega;\slot)\,\rd\mu -  \int f\,\rd\mu \int g\,\rd\mu\right| 
\\
& \quad
\leq \| f\|_\infty B_-(\sigma^{n+k}\omega;\vartheta_-,g)\,(\theta(\vartheta_-))^{k/4} +  2C_\mathrm{p}\|f\|_\infty\|g\|_\infty\, e^{-k/4\chi}+ 2K_+(\sigma^n\omega;\vartheta_+,f)\|g\|_\infty\, \vartheta_+^{k/2-1}
\\
& \quad
\leq
\left(\prod_{i=0}^n\|f_i\|_\infty\right) \left\{C\prod_{i=0}^m\|g_{n+k+i}\|_\infty + K_-(\sigma^{n+k}\omega;\vartheta_-,g)\,\vartheta_-^{-1} \right\}(\theta(\vartheta_-))^{k/4} 
\\
& \qquad+  2C_\mathrm{p}\left(\prod_{i=0}^n\|f_i\|_\infty\right) \left(\prod_{i=0}^m\|g_{n+k+i}\|_\infty\right) e^{-k/4\chi}
+ 2K_+(\sigma^n\omega;\vartheta_+,f)\left(\prod_{i=0}^m\|g_{n+k+i}\|_\infty\right) \vartheta_+^{k/2-1}.
\end{split}
\eeqn
Inserting~\eqref{eq:K_+_temp},~\eqref{eq:K_-_temp}, and $\theta(\vartheta_-) = \max(\zeta,\vartheta_-^{1/2})\in(0,1)$ proves the theorem.
\end{proof}

Now, we specialize to the actual case of interest, which is related to proving convergence of the characteristic function of the vector-valued process in Theorem~\ref{thm:ASIP}. Thus, below, the functions $f_i$ and $g_{n+k+i}$ are going to be uniformly bounded by $1$, as well as have the same rate and the same constant in the dynamical H\"older continuity conditions, as follows.

\begin{cor}\label{cor:multiple_bound}
For a sufficiently small $\ve>0$, there exist uniform constants $C_\mathrm{p}>0$, $C>0$, $\zeta>0$, and $\chi>0$ such that the following holds.
Let $\omega\in\Omega_\mathrm{sat}$, $n\geq 0$, $m\geq 0$, $k\geq 0$, $\vartheta\in (0,1)$, and $K\geq 0$ be fixed. Assume that $f_i\in\cH_+(\sigma^i\omega;\vartheta)$ with $K_+(\sigma^i\omega;\vartheta,f_i)\leq K$ and $\|f_i\|_\infty\leq 1$ for $0\leq i\leq n$, as well as $g_{n+k+i}\in\cH_-(\sigma^{n+k+i}\omega;\vartheta)$ with $K_-(\sigma^{n+k+i}\omega;\vartheta,g_{n+k+i})\leq K$ and $\|g_{n+k+i}\|_\infty\leq 1$ for $0\leq i\leq m$. With $F$ and $G$ as in~\eqref{eq:FG_def},
we have the bound
\beqn
\left|\operatorname{Cov}_\mu{(F,G)}\right| = \left|\int FG\,\rd\mu-\int F\,\rd\mu\int G\,\rd\mu\right| \leq A\lambda^k,
\eeqn
where $A=C +  2C_\mathrm{p}+3K(1-\vartheta)^{-1}\vartheta^{-1}> 0$ and $\lambda= \left(\max(\zeta,\vartheta^{1/2},e^{-1/\chi})\right)^{1/4} \in(0,1)$.
\end{cor}

\begin{proof}
From Theorem~\ref{thm:multiple_bound}, we easily obtain
\beqn
\begin{split}
\left|\operatorname{Cov}_\mu{(F,G)}\right|
\leq  \left\{C +  K\sum_{i=0}^n  \vartheta^{i}\,\vartheta^{-1} \right\} \left(\max(\zeta,\vartheta^{1/2})\right)^{k/4} 
+  2C_\mathrm{p} e^{-k/4\chi}
+ 2K\sum_{i=0}^n  \vartheta^{n-i} \vartheta^{k/2-1}.
\end{split}
\eeqn
The result follows from this, because $\sum_{i=0}^n\vartheta^i = \sum_{i=0}^n  \vartheta^{n-i}\leq (1-\vartheta)^{-1}$.
\end{proof}

Notice that the above multiple correlation bounds were deduced directly from the pair correlation bound in Theorem~\ref{thm:pair_bound}.


\subsection{Characteristic function bounds}\label{sec:ASIP_proof}
Let us emphasize that our pair and multiple correlation bounds hold for \emph{all} sequences $\omega\in\Omega_\mathrm{sat}$, without any statistics required on the sequences. We will show below that an application of the bounds on mixing, asymptotically mean stationary, random sequences yields the vector-valued almost sure invariance principle.

The next theorem follows when \cite{Gouezel_2010} is applied to the special case of bounded processes. 
\begin{thm}[Gou\"ezel \cite{Gouezel_2010}]\label{thm:Gouezel}
Let $d$ be a positive integer and consider a sequence $(\bfA_n)_{n\geq 0}$ of $\bR^d$-valued random variables which is centered and bounded. Given integers $n>0$, $m>0$, $0\leq b_1<b_2<\dots<b_{n+m+1}$, $k\geq 0$, and vectors $\bft_1,\dots,\bft_{n+m}\in\bR^d$, set
\beq\label{eq:X_array}
X_{n,m}^{(k)} = \sum_{j=n}^m \bft_j\cdot\sum_{i=b_j+k}^{b_{j+1}-1+k}\bfA_{i}
\eeq
for brevity.  Now, suppose there exist constants $t>0$, $C>0$, and $c>0$ such that
\beq\label{eq:Gouezel}
\left|\E\!\left(e^{\imag X_{1,n}^{(0)}+\imag X_{n+1,n+m}^{(k)}}\right)-\E\!\left(e^{\imag X_{1,n}^{(0)}}\right)\! \E\!\left(e^{\imag X_{n+1,n+m}^{(k)}}\right)\right|
\leq Ce^{-ck}\!\left(1+\max_{1\leq j\leq n+m}|b_{j+1}-b_j|\right)^{C(n+m)}
\eeq
holds for all choices of the numbers $n$, $m$, $b_j$, $k>0$, and all vectors $\bft_j$ satisfying $|\bft_j|<t$.
\medskip
\\
{\bf(I)} If $(\bfA_n)_{n\geq 0}$ is stationary, there exists such a symmetric, semi-positive-definite, $d\times d$ matrix~$\bfSigma^2$ that the statements (2)--(4) listed as consequences in Theorem~\ref{thm:ASIP} hold true.
\medskip
\\
\noindent{\bf(II)} If $(\bfA_n)_{n\geq 0}$ is non-stationary, and if there exists such a $d\times d$ matrix~$\bfSigma^2$ that, for any $\alpha>0$,
\beq\label{eq:linear_covariance}
\sup_{\ell\geq 0,n\geq 1}\, n^{-\alpha}\left\|\rE\!\left(\sum_{k=\ell}^{\ell+n-1}\bfA_k \otimes \sum_{k=\ell}^{\ell+n-1}\bfA_k\right) - n\bfSigma^2\right\| <\infty
\eeq
is satisfied, then items (3) and (4) in Theorem~\ref{thm:ASIP} hold true. (Here $\|\slot\|$ denotes matrix norm.)
\end{thm}

Denote $\Omega_+ = (B_\ve(\mathbf{0}))^\bN$ and let $\bP$ be a probability distribution on it, as in the statement of Theorem~\ref{thm:ASIP}. Notice that we can embed $\Omega_+$ in $\Omega_\mathrm{sat}$ canonically: Given a sequence $\omega\in\Omega_+$, we identify it with the sequence $\bar\omega\in\Omega_\mathrm{sat}$ given by $\bar\omega_i = \omega_i$ for all $i\geq 0$ and $\bar\omega_i = \omega_0$ for all $i< 0$. For this reason, we will not make an explicit distinction with one- and two-sided sequences in the following argument. (That is, $\omega\in\Omega_+$ is picked randomly according to $\bP$ and fixed, after which it is augmented to an element of $\Omega_\mathrm{sat}$, still denoted $\omega$, which is then used in the subsequent computations.)

In our application of Theorem~\ref{thm:Gouezel}, given $\omega\in\Omega_+$, we set
$
\bfA_n(\omega,x) = \bff(\omega_n,\cF_n(\omega,x))
$
and define $X_{n,m}^{(k)}(\omega,x)$ as in~\eqref{eq:X_array}. Notice that we can write
\beqn
F(\omega,x)=\exp\!\left(
 \imag X_{1,n}^{(0)}(\omega,x)\right) = \prod_{i=b_1}^{b_{n+1}-1} f_i(\omega_i,\cF_i (\omega,x)),
\eeqn
where
$
f_i(\omega_i,x) = \exp\!\left(
 \imag \bft_j\cdot\bff(\omega_i,x)
\right)
$
with $j\in\{1,\dots,n\}$ chosen so that $b_j\leq i<b_{j+1}$.
Similarly,
\beqn
G(\omega,x)=\exp\!\left(
\imag X_{n+1,n+m}^{(k)}(\omega,x)
\right)
= 
\prod_{i=b_{n+1}+k}^{b_{n+m+1}-1+k} g_{i}(\omega_i,\cF_{i}(\omega,x)),
\eeqn
where
$
g_i(\omega_i,x) = \exp\!\left(
 \imag \bft_j\cdot\bff(\omega_i,x)\right)
$
with $j\in\{n+1,\dots,n+m\}$ chosen so that $b_j+k\leq i<b_{j+1}+k$. Thus, our goal is to show that the covariance of $F$ and $G$ with respect to the measure~$\bP\otimes\mu$ has a uniform upper bound like the right side of \eqref{eq:Gouezel}. Notice already that
\beq\label{eq:cov_identity}
\operatorname{Cov}_{\bP\otimes\mu}{(F,G)} = \int \operatorname{Cov}_\mu{(F,G)}\,\rd\bP + \operatorname{Cov}_\bP\left(\int F\,\rd\mu\,,\int G\,\,\rd\mu\right).
\eeq
Both $(\omega,x)\mapsto\bff(\omega_n,x)$ and $(\omega,x)\mapsto\cF_n(\omega,x)$ are measurable for every $n\geq 0$. (The first one of the two is an assumption, and the second is a consequence of the continuity statement after Lemma~\ref{lem:map_cont}.) In particular, $(\omega,x)\mapsto F(\omega,x)$ and $(\omega,x)\mapsto G(\omega,x)$ are measurable and absolutely bounded by $1$. Thus, all the integrals (including the covariances) in~\eqref{eq:cov_identity} are defined.

Owing to the H\"older continuity assumption on $\bff$, Lemma~\ref{lem:Holder_dyn_Holder} shows that there exist constants $\vartheta\in(0,1)$ and $\bar K\geq 0$ such that, for every $i\geq 0$, every vector component of the function $\bff(\omega_i,\slot)$ belongs to $\cH_-(\sigma^i\omega;\vartheta)\cap\cH_+(\sigma^i\omega;\vartheta)$ and
\beqn
| \bff(\omega_i,x) - \bff(\omega_i, y) | \leq \bar K \vartheta^{s_{\pm}(\sigma^i\omega;x,y)} \qquad \forall\,x,y\in W_{\pm,\alpha}^{(i)}\quad\forall\,\alpha\in\cA^{(i)}_{\pm}.
\eeqn
(In each condition we either choose ``$+$'' everywhere or ``$-$'' everywhere.) Assuming $|\bft_j|<t$ for all $j\geq 1$, the bound $|e^{\imag a}-e^{\imag b}|\leq |a-b|$ for $a,b\in\bR$ implies
\beqn
|f_i(\omega_i,x) -f_i(\omega_i,y) |,|g_i(\omega_i,x) -g_i(\omega_i,y) | \leq 
t\, | \bff(\omega_i,x) - \bff(\omega_i, y) |.
\eeqn
We can thus apply the multiple correlation bound of Corollary~\ref{cor:multiple_bound}, defining $K = t\bar K$. The small formal difference that here the product in the expression of~$F$ starts from time~$b_1$ instead of~$0$ is abolished by regarding momentarily $f_i \equiv 1$ for each $i=0,\dots,b_1-1$. Consequently,
$
|\operatorname{Cov}_\mu{(F,G)}| 
\leq A\lambda^k,
$
with the uniform, non-random, constants $A$ and $\lambda$. In particular,
\beqn
\left|\int \operatorname{Cov}_\mu{(F,G)}\,\rd\bP\right|\leq A\lambda^k.
\eeqn

The maps $\bar F:\omega\mapsto \int F(\omega,\slot)\,\rd\mu$ and $\bar G:\omega\mapsto \int G(\omega,\slot)\,\rd\mu$ are measurable. What is more, $\bar F(\omega)$ only depends on $\omega_i$ with $b_1\leq i<b_{n+1}$, while $\bar G(\omega)$ only depends on $\omega_i$ with $b_{n+1}+k\leq i<b_{n+m+1}-1+k$. Hence, if the sequence $(\omega_i)_{i\geq 0}$ is random and $\fF_a^b$, with $0\leq a\leq b$ integers, stands for the sigma-algebra generated by~$(\omega_i)_{i=a}^b$, then  $\bar F(\omega)$ and $\bar G(\omega)$ are, respectively, $\fF_0^{b_{n+1}-1}$- and $\fF_{b_{n+1}+k}^\infty$-measurable complex-valued random variables. By the rho-mixing assumption~{\bf(A1)}, and the fact that $|\bar F|\leq 1$ and $|\bar G|\leq 1$, the identity in~\eqref{eq:rho} yields
\beqn
\bigl|\operatorname{Cov}_\bP{(\bar F,\bar G)}\bigr| \leq \rho(k) \leq Be^{-ck}.
\eeqn

By~\eqref{eq:cov_identity}, we conclude that, in the notation of Theorem~\ref{thm:Gouezel},
\beqn
\begin{split}
& \left|\operatorname{Cov}_{\bP\otimes\mu}{(F,G)}\right|\leq A\lambda^k + Be^{-ck}
\end{split}
\eeqn
holds uniformly for all choices of the numbers $n$, $m$, $b_j$, $k>0$, and of the vectors $\bft_j$ satisfying $|\bft_j|<t$. The obtained bound is obviously stronger than Gou\"ezel's condition~\eqref{eq:Gouezel}. Hence, the proof of Theorem~\ref{thm:ASIP} in the stationary case is complete.

\subsection{Convergence of covariance}\label{sec:proof_covariance}
In this section we establish uniform convergence of the covariance in the sense of~\eqref{eq:linear_covariance}, which is needed in the non-stationary case. We denote
\beqn
\bfS_{n,\ell} = \sum_{k=\ell}^{\ell+n-1}\bfA_k \quad\text{and}\quad \bfS_n = \bfS_{n,0}.
\eeqn

\begin{lem}\label{lem:linear_covariance}
The matrix~$\bfSigma^2$ in~\eqref{eq:covariance_formula} is well defined. It satisfies the bound in~\eqref{eq:linear_covariance}, for any $\alpha>0$.
\end{lem}
\begin{proof}
By $\mu$-invariance,
\beqn
\rE(\bfS_{n,\ell}\otimes \bfS_{n,\ell}) = \iint\bigl(\bfS_{n}\otimes \bfS_{n}\bigr)(\sigma^\ell\omega,x)\,\rd\mu(x)\,\rd\bP(\omega).
\eeqn
For any $\omega$,
\beqn
\begin{split}
& \int\bigl(\bfS_{n}\otimes \bfS_{n}\bigr)(\omega,\slot)\,\rd\mu 
\\
& \qquad\quad = \sum_{k=0}^{n-1} \int \bigl(\bfA_k\otimes\bfA_k\bigr)(\omega,\slot)\,\rd\mu + \sum_{0\leq j<k\leq n-1} \int \bigl(\bfA_j\otimes\bfA_k + \bfA_k\otimes\bfA_j\bigr)(\omega,\slot)\,\rd\mu
\\
& \qquad \quad =
\sum_{k=0}^{n-1} \int \bigl(\bfA_0\otimes\bfA_0\bigr)(\sigma^k\omega,\slot)\,\rd\mu + \sum_{0\leq j<k\leq n-1} \int \bigl(\bfA_0\otimes\bfA_{k-j} + \bfA_{k-j}\otimes\bfA_0\bigr)(\sigma^j\omega,\slot)\,\rd\mu
\\
& \qquad \quad =
\sum_{j=0}^{n-1} \int \bigl(\bfA_0\otimes\bfA_0\bigr)(\sigma^j\omega,\slot)\,\rd\mu + \sum_{m=1}^{n-1}\sum_{j=0}^{n-1-m} \int \bigl(\bfA_0\otimes\bfA_{m} + \bfA_{m}\otimes\bfA_0\bigr)(\sigma^j\omega,\slot)\,\rd\mu.
\end{split}
\eeqn
Denoting
\beqn
\bfV_0(\omega) = \int \bigl(\bfA_0\otimes\bfA_0\bigr)(\omega,x)\,\rd\mu(x)
\eeqn
and
\beqn
\bfV_m(\omega) = \int \bigl(\bfA_0\otimes\bfA_{m} + \bfA_{m}\otimes\bfA_0\bigr)(\omega,x)\,\rd\mu(x), \quad m\geq 1,
\eeqn
as well as using the notation introduced in~\eqref{eq:mean},
we thus have
\beqn
\begin{split}
 \rE(\bfS_{n,\ell}\otimes \bfS_{n,\ell}\bigr) & =
\sum_{m=0}^{n-1}\sum_{j=0}^{n-1-m} \bE (\bfV_m\circ\sigma^{\ell+j})
 = \sum_{m=0}^{n-1}(n-m) \, \langle \bfV_m\circ\sigma^\ell \rangle_{n-m} 
\\
&=  n \sum_{m=0}^{n-1}\, \langle \bfV_m\circ\sigma^\ell \rangle_{n-m} - \sum_{m=0}^{n-1} m\, \langle \bfV_m\circ\sigma^\ell \rangle_{n-m} .
\end{split}
\eeqn

By Theorem~\ref{thm:pair_bound}, there exist $C\geq 0$ and $\beta>0$ such that
\beq\label{eq:V_bound}
\sup_{\omega\in\Omega_+}\|\bfV_m(\omega)\| \leq C e^{-\beta m}, \quad m\geq 0.
\eeq
Therefore, $\|\langle \bfV_m\circ\sigma^\ell \rangle_k\| \leq C e^{-\beta m}$ for all $m,\ell\geq 0$ and $k\geq 1$, which yields
\beqn
\left\| \sum_{m=0}^{n-1} m\, \langle \bfV_m \circ\sigma^\ell\rangle_{n-m} \right\| \leq  C\sum_{m=0}^{n-1} m\, e^{-\beta m} \leq C\sum_{m=0}^{\infty} m\, e^{-\beta m} = C',
\eeqn
or
\beq\label{eq:prelim_cov_bound}
\left\|\rE(\bfS_{n,\ell}\otimes \bfS_{n,\ell})  - n \sum_{m=0}^{n-1} \, \langle \bfV_m\circ\sigma^\ell \rangle_{n-m} \right\| \leq C',\qquad \ell\geq 0,\,n\geq 1.
\eeq

Notice that $\bfV_m$ is bounded and $\fF_0^m$-measurable. By assumption~{\bf(A2)}, there exist the limits
\beqn
\langle \bfV_m \rangle_\infty = \lim_{k\to\infty}\langle \bfV_m \circ\sigma^\ell\rangle_k, \qquad \ell\geq 0,\,m\geq 0.
\eeqn
Recall that in~\eqref{eq:covariance_formula} we have claimed that the desired covariance matrix is given by
\beq\label{eq:covariance}
\bfSigma^2 \equiv \sum_{m=0}^{\infty} \, \langle \bfV_m \rangle_\infty.
\eeq
Certainly by~\eqref{eq:V_bound} the series on the right converges in norm and yields a well-defined matrix.
\begin{lem}\label{lem:AMS}
For any bounded and $\fF$-measurable function $\bfg:(B_\ve(\mathbf{0}))^\bN\to\bR^d$, 
\beqn
\langle \bfg\circ\sigma^\ell \rangle_\infty = \langle \bfg \rangle_\infty,\qquad\ell\geq 0,\,m\geq 0.
\eeqn
\begin{proof}
We only check the claim for $\ell=1$, the general case being similar. Clearly $\bfg\circ \sigma$ is bounded and measurable ($\fF$ being the product Borel sigma-algebra), and
\beqn
\langle \bfg \circ \sigma\rangle_k - \langle \bfg \rangle_k = \frac{1}{k}\bigl(\bE(\bfg \circ \sigma)- \bE(\bfg)\bigr) \to 0\quad\text{as $k\to\infty$}.
\eeqn
Since $\langle \bfg \circ \sigma\rangle_k \to \langle \bfg \circ \sigma\rangle_\infty$, the claim follows.
\end{proof}
\end{lem}

Continuing with~\eqref{eq:prelim_cov_bound}, if $\beta^{-1}\log n<n$,
\beqn
\sum_{m=0}^{n-1} \, \langle \bfV_m \circ \sigma^\ell \rangle_{n-m} = \sum_{m=0}^{\beta^{-1}\!\log n-1} \, \langle \bfV_m \circ \sigma^\ell  \rangle_{n-m} + \sum_{m=\beta^{-1}\!\log n}^{n-1} \, \langle \bfV_m \circ \sigma^\ell  \rangle_{n-m},
\eeqn
where
\beqn
\left\| \sum_{m=\beta^{-1}\!\log n}^{n-1} \, \langle \bfV_m \circ \sigma^\ell  \rangle_{n-m} \right\| \leq C \sum_{m=\beta^{-1}\!\log n}^{n-1} e^{-\beta m} \leq \frac{C}{n} \sum_{m=0}^{\infty} e^{-\beta m}.
\eeqn
In other words, we have shown that
\beqn
\left\|\rE(\bfS_{n,\ell}\otimes \bfS_{n,\ell})  - n \sum_{m=0}^{\beta^{-1}\!\log n-1} \, \langle \bfV_m \circ \sigma^\ell  \rangle_{n-m} \right\| \leq 2 C',\qquad\beta^{-1}\log n<n.
\eeqn
Next, recalling \eqref{eq:covariance} and Lemma~\ref{lem:AMS}, we decompose
\beqn
\sum_{m=0}^{\beta^{-1}\!\log n-1} \, \langle \bfV_m\circ \sigma^\ell  \rangle_{n-m} 
= \bfSigma^2 - \sum_{m=\beta^{-1}\!\log n}^\infty \, \langle \bfV_m  \rangle_\infty + \sum_{m=0}^{\beta^{-1}\!\log n-1} \, \Bigl(\langle \bfV_m\circ \sigma^\ell  \rangle_{n-m} -\langle \bfV_m  \rangle_\infty \Bigr).
\eeqn
The middle term on the right can be bounded using again~\eqref{eq:V_bound}. Indeed,
\beqn
\left\| \sum_{m=\beta^{-1}\!\log n}^{\infty} \, \langle \bfV_m \rangle_\infty \right\| \leq C \sum_{m=\beta^{-1}\!\log n}^\infty e^{-\beta m} \leq \frac{C'}{n}.
\eeqn
Hence, using assumption~{\bf(A2)}, we can compute
\beqn
\begin{split}
&\left\|\rE(\bfS_{n,\ell}\otimes \bfS_{n,\ell})  - n \bfSigma^2 \right\|  \leq 3 C' + n \left\|\sum_{m=0}^{\beta^{-1}\!\log n-1} \, \Bigl(\langle \bfV_m\circ\sigma^\ell \rangle_{n-m} -\langle \bfV_m \rangle_\infty \Bigr)\right\|
\\
& \qquad \qquad
\leq 3 C' + n \sum_{m=0}^{\beta^{-1}\!\log n-1} C_m\, r_{n-m}\, \| \bfV_m \|_\infty
\leq 3 C' + C n \sum_{m=0}^{\beta^{-1}\!\log n-1} C_m\, r_{n-m}\, e^{-\beta m}
\\
& \qquad\qquad
 \leq 3 C' + C D_\beta\log n
,\qquad \beta^{-1}\log n<n.
\end{split}
\eeqn
Finally,~\eqref{eq:prelim_cov_bound} yields also the crude bound
\beqn
\begin{split}
\left\|\rE(\bfS_{n,\ell}\otimes \bfS_{n,\ell})  - n \bfSigma^2 \right\| 
&
\leq C' + n \left\|\bfSigma^2  -  \sum_{m=0}^{n-1} \, \langle \bfV_m\circ\sigma^\ell \rangle_{n-m} \right\| \leq 
C' + n (\left\|\bfSigma^2\right\| + C') 
\\
&\leq 2 C'(1+ \beta^{-1} \log n)
,\qquad \beta^{-1}\log n\geq n.
\end{split}
\eeqn
Collecting the bounds in the two regimes $\beta^{-1}\log n< n$ and $\beta^{-1}\log n\geq n$, we see that there exists a constant $C''\geq 0$ such that
\beqn
\left\|\rE(\bfS_{n}\otimes \bfS_{n})  - n \bfSigma^2 \right\| \leq 3C'+ C''\log n, \quad n\geq 1.
\eeqn
This finishes the proof of Lemma~\ref{lem:linear_covariance}.
\end{proof}

An application of Lemma~\ref{lem:linear_covariance} with part~{\bf(II)} of Theorem~\ref{thm:Gouezel} finishes the proof of Theorem~\ref{thm:ASIP}.\qed


\begin{remark}
Related to the special case of Theorem~\ref{thm:fixed_ASIP}, the above proof contains the interesting fact that a pair correlation bound (for a well-chosen class of dynamically H\"older observables) alone implies the vector-valued almost sure invariance principle for Sinai billiards with fixed scatterers. This extends the analogous result of \cite{Stenlund_2010} about the central limit theorem for Sinai billiards. Before that, the otherwise classical central limit theorem for Sinai billiards had been obtained via the new method of multiple (as opposed to just pair) correlation bounds in \cite{Chernov_2006,ChernovMarkarian_2006}. The author has learned that also P\`ene \cite{Pene_2005} has proved the central limit theorem for Sinai billiards using correlation functions without, however, pursuing finer limit theorems.

\end{remark}


\subsection{Proof of Lemma~\ref{lem:coboundary}}\label{sec:proof_coboundary}

Let us write $\bP_k$ for the probability measure on $\bigl((B_\ve(\mathbf{0}))^\bN,\fF\bigr)$ given by $\bP_k(A) =  \tfrac{1}{k}\sum_{j=0}^{k-1} \bP(\sigma^{-j} A)$ for all $A\in\fF$. By assumption~{\bf(A2)}, $\bP_k(A) = \langle 1_A\rangle_k \to\langle 1_A\rangle_\infty$ as $k\to\infty$. The Vitali--Hahn--Saks theorem now states that the map $\bar\bP:\fF\to [0,1]:A\mapsto \langle 1_A\rangle_\infty$ is a probability measure. The measure $\bar \bP$ is called the \emph{stationary mean} of $\bP$, which is justified by the simple but crucial observation made in Lemma~\ref{lem:AMS} of the preceding section that~$\bar\bP$ --- unlike~$\bP$ --- is invariant for the left shift~$\sigma$. Because $\mu$ is invariant for each of the billiard maps, the measure $\bar\rP = \bar\bP\otimes \mu$ is therefore invariant for the skew-product map $\Phi$ defined in the statement of the lemma. 

Observe that, writing $\bar\rE$ for the expectation relative to $\bar \rP$,~\eqref{eq:covariance_formula} reads
\beq\label{eq:cov_new}
\begin{split}
\bfSigma^2 & = \bar\rE(\bfA_0\otimes\bfA_0)+ \sum_{m=1}^{\infty} \bar\rE (\bfA_0\otimes\bfA_{m} + \bfA_{m}\otimes\bfA_0).
\end{split}
\eeq
Here $\bfA_n(\omega,x) = \bff(\omega_n,\cF_n(\omega,x)) = \bff(\Phi^n(\omega,x))$ with the understanding that $\bff(\omega,x)\equiv \bff(\omega_0,x)$.  The following is a complementary result to Lemma~\ref{lem:linear_covariance}. It compares the covariance of the sum $\bfS_n = \sum_{i=0}^{n-1}\bfA_i = \sum_{i=0}^{n-1} \bff\circ\Phi^i$, computed with respect to $\bar\rP$ instead of $\rP$, to $n\bfSigma^2$. Taking advantage of the $\Phi$-invariance of $\bar\rP$ yields a match better than what one would naively infer from~\eqref{eq:linear_covariance}:
\begin{lem}
There exists a constant $C'\geq 0$ such that
\beq\label{eq:cov_new_bound}
\sup_{n\geq 0}\left\| \bar \rE\! \left(\bfS_n\otimes \bfS_n \right) - n\bfSigma^2 \right\| \leq C'.
\eeq
\end{lem}
\begin{proof}
We start with the elementary identity
\beqn
\begin{split}
\bar\rE\! \left(\bfS_n\otimes \bfS_n \right) 
& =
n\,\bar\rE(\bfA_0\otimes\bfA_0)+ \sum_{m=1}^{n-1} (n-m) \,\bar\rE (\bfA_0\otimes\bfA_{m} + \bfA_{m}\otimes\bfA_0),
\end{split}
\eeqn
which is where we need invariance. Recalling~\eqref{eq:cov_new}, we have
\beqn
\begin{split}
\bar\rE\! \left(\bfS_n\otimes \bfS_n \right) - n\bfSigma^2 
& = \sum_{m=1}^{\infty} a_n(m) \,\bar\rE (\bfA_0\otimes\bfA_{m} + \bfA_{m}\otimes\bfA_0)
\end{split}
\eeqn
where $a_n(m) = -m$ for $1\leq m<n$ and $a_n(m) = -n$ for $m\geq n$. Because $|a_n(m)|\leq m$,
\beqn
\left \|\bar\rE\! \left(\bfS_n\otimes \bfS_n \right) - n\bfSigma^2 \right\| \leq \sum_{m=1}^{\infty} m \left\| \bar\rE (\bfA_0\otimes\bfA_{m} + \bfA_{m}\otimes\bfA_0) \right\| .
\eeqn
In the notation of Section~\ref{sec:proof_covariance},
\beqn
\bar\rE (\bfA_0\otimes\bfA_{m} + \bfA_{m}\otimes\bfA_0) = \lim_{k\to\infty}\frac1k\sum_{\ell = 0}^{k-1} \bE(\bfV_m\circ\sigma^\ell).
\eeqn
The uniform bound in~\eqref{eq:V_bound} yields $\|\bE(\bfV_m\circ\sigma^\ell)\|\leq Ce^{-\beta m}$ and hence the result.
\end{proof}
Recall the $\mu$-average of $\bff$ vanishes identically.
Given a vector $\bfv\in\bR^d$, we define $f_\bfv = \bfv^\rT \bff$. Since $\bfv^\rT(\bfA_m\otimes\bfA_n)\bfv = \bfv^\rT(\bfA_n\otimes\bfA_m)\bfv = f_\bfv\circ\Phi^n \cdot f_\bfv\circ\Phi^m$ for all $m,n\geq 0$,~\eqref{eq:cov_new_bound} gives
\beq\label{eq:cov_new_bound2}
\left|\operatorname{Var}_{\bar\rP}\!\left(\sum_{k=0}^{n-1} f_\bfv\circ\Phi^k \right) - n\,\bfv^\rT  \bfSigma^2 \bfv\right|  = \left|\bfv^\rT \bigl(\bar \rE(\bfS_{n}\otimes \bfS_{n})  - n \bfSigma^2 \bigr)\bfv \right| \leq C'|\bfv|^2
\eeq
uniformly for $n\geq 1$. 

Suppose $\bfSigma^2$ is degenerate. In other words, there exists a vector $\bfv\in\bR^d$ such that $\bfv^\rT\bfSigma^2 \bfv = 0$.  By the bound above,
$
\operatorname{Var}_{\bar\rP}\!\left(\sum_{k=0}^{n-1} f_\bfv\circ\Phi^k \right) \leq C'|\bfv|^2
$
uniformly. Owing to the~$\Phi$-invariance of~$\bar\rP$, we can therefore apply~\cite[Lemma~1]{Leonov_1961} (see also~\cite{Robinson_1960}) and conclude that~$f_\bfv$ must be an~$L^2(\bar \rP)$-coboundary. In other words, there exists $g\in L^2(\bar \rP)$ such that $f_\bfv = g-g\circ\Phi$ holds $\bar\rP$-almost-everywhere. Since
\beq\label{eq:L2_norms}
\|g\|_{L^2(\bar \rP)}^2 = \lim_{k\to\infty}\frac1k \sum_{j=0}^{k-1} \int | g(\sigma^j\omega,x) |^2\,\rd \rP(\omega,x) 
,
\eeq
the claim of the lemma in one direction follows. 

To prove the claim in the other direction, suppose $f_\bfv = g-g\circ\Phi$ for some $g$ that satisfies the condition given in the lemma. Then $g\in L^2(\bar\rP)$ by \eqref{eq:L2_norms}, and $\sum_{i=0}^{n-1}f_\bfv\circ\Phi^i = g - g\circ\Phi^{n}$ holds $\bar\rP$-almost-everywhere. This immediately gives
\beqn
\operatorname{Var}_{\bar\rP}\!\left(\sum_{k=0}^{n-1} f_\bfv\circ\Phi^k \right) = \| g-g\circ\Phi^n \|_{L^2(\bar\rP)}^2 \leq 4\| g \|_{L^2(\bar\rP)}^2 .
\eeqn
Combining the bound with \eqref{eq:cov_new_bound2} we get $\bfv^\rT  \bfSigma^2 \bfv \leq n^{-1}\bigl(C'|\bfv|^2 + 4\| g \|^2_{L^2(\bar\rP)}\bigr)$ for all $n\geq 1$, which is only possible if $\bfSigma^2$ is degenerate in the direction of $\bfv$. This completes the proof of Lemma~\ref{lem:coboundary}.\qed




\bibliography{RB}{}
\bibliographystyle{plainurl}


\end{document}